\documentclass[a4paper,reqno,11pt]{amsart}

\numberwithin{equation}{section}

\usepackage[abbrev]{amsrefs}

\usepackage{amssymb}

\usepackage{enumerate}

\usepackage[T1]{fontenc}

\usepackage{microtype}

\usepackage{libertine}

\newtheorem{proposition}{Proposition}[section]
\newtheorem{theorem}[proposition]{Theorem}
\newtheorem{lemma}[proposition]{Lemma}
\newtheorem{corollary}[proposition]{Corollary}
\theoremstyle{definition}
\newtheorem{definition}{Definition}[section]
\newtheorem{openproblem}{Open Problem}[section]
\theoremstyle{remark}
\newtheorem{remark}{Remark}[section]
\newtheorem{example}{Example}[section]
\newcommand{\abs}[1]{\lvert#1\rvert}
\newcommand{\bigabs}[1]{\bigl\lvert#1\bigr\rvert}

\newcommand{\dualprod}[2]{\langle #1 , #2 \rangle}
\newcommand{\st}{\: :\:}
\newcommand{\norm}[1]{\lVert#1\rVert}
\newcommand{\R}{\mathbf{R}}
\newcommand{\Lin}{\mathcal{L}}
\newcommand{\N}{\mathbf{N}}
\newcommand{\quat}{\mathbf{H}}

\DeclareMathOperator{\supp}{supp}
\DeclareMathOperator{\Div}{div}
\DeclareMathOperator{\Curl}{curl}
\DeclareMathOperator{\adj}{adj}
\DeclareMathOperator{\id}{id}
\DeclareMathOperator{\tr}{tr}
\DeclareMathOperator{\realpart}{Re}

\title[Limiting Sobolev inequalities and canceling operators]{Limiting Sobolev inequalities for vector fields and canceling linear differential operators}
\author{Jean Van Schaftingen}
\address{Universit\'e catholique de Louvain\\
Institut de Recherche en Math\'ematique et Physique (IRMP)\\
Chemin du Cyclotron 2 bte L7.01.01\\
1348 Louvain-la-Neuve\\
Belgium}
\email{Jean.VanSchaftingen@uclouvain.be}
\date{\today}
\begin{document}

\begin{abstract}
The estimate
\[
 \norm{D^{k-1}u}_{L^{n/(n-1)}} \le \norm{A(D)u}_{L^1}
\]
is shown to hold if and only if \(A(D)\) is elliptic and canceling.
Here  \(A(D)\) is a homogeneous linear differential operator \(A(D)\) of order \(k\) on \(\R^n\) from a vector space \(V\) to a vector space \(E\). The operator \(A(D)\) is defined to be canceling if 
\[
  \bigcap_{\xi \in \R^n \setminus \{0\}} A(\xi)[V]=\{0\}.
\]
This result implies in particular the classical Gagliardo--Nirenberg-Sobolev inequality, the Korn--Sobolev inequality and Hodge--Sobolev estimates for differential forms due to J.\thinspace Bourgain and H.\thinspace Brezis.
In the proof, the class of cocanceling homogeneous linear differential operator \(L(D)\) of order \(k\) on \(\R^n\) from a vector space \(E\) to a vector space \(F\) is introduced. It is proved that \(L(D)\) is cocanceling if and only if for every \(f \in L^1(\R^n; E)\) such that \(L(D)f=0\), one has \(f \in \dot{W}^{-1, n/(n-1)}(\R^n; E)\). The results extend to fractional and Lorentz spaces and can be strengthened using some tools of J.\thinspace Bourgain and H.\thinspace Brezis.
\end{abstract}

\maketitle

\tableofcontents

\section{Introduction}

\subsection{Norms on vector valued homogeneous Sobolev spaces}
Given \(n  \ge 1\), \(k \in \N\), \(p \ge 1\) and a finite-dimensional vector space \(V\), the homogeneous Sobolev space \(\dot{W}^{k, p}(\R^n; V)\) can be characterized as the completion of the space of smooth vector fields \(C^\infty_c(\R^n; V)\) under the norm defined for \(u \in C^\infty_c(\R^n; V)\) by 
\[
  \norm{D^ku}_{L^p}=\Bigl( \int_{\R^n} \abs{D^k u}^p \Bigr)^\frac{1}{p}.
\]

When \(\dim V > 1\), one can wonder whether the norm can be estimated by a quantity involving only some components of the derivative. More precisely, assume that \(A(D)\) is a homogeneous differential operator of order \(k\) on \(\R^n\) from \(V\) to another finite-dimensional vector space \(E\), that is there exist linear maps \(A_\alpha \in \Lin(V;E)\) with \(\alpha \in \N^n\) and \(\abs{\alpha}=k\) such that for every \( u \in C^\infty(\R^n; V)\), 
\[
  A(D) u = \sum_{\substack{\alpha \in \N^n\\ \abs{\alpha}=k}} A_\alpha (\partial^\alpha u) \in C^\infty(\R^n; E).
\]
One can ask the question whether the norms defined for \(u \in C^\infty_c(\R^n; V) \) by \(\norm{D^ku}_{L^p}\) and \(\norm{A(D)u}_{L^p}\) are equivalent.

When \(p > 1\), the answer is given by the classical result
\begin{theorem}[A.\thinspace P.\thinspace Calder\'on and A.\thinspace Zygmund, 1952 \cite{CZ1952}]	
\label{theoremEllipticLp}
Let \(1 < p < \infty\) and \(A(D)\) be a homogeneous differential operator of order \(k\) on \(\R^n\) from \(V\) to \(E\). The estimate
\[
  \norm{D^{k}u }_{L^p} \le C\norm{A(D) u}_{L^p},
\]
holds for every \(u \in C^\infty_c(\R^n; V)\)
if and only if \(A(D)\) is elliptic.
\end{theorem}

Here and in the sequel the constant \(C\) is understood to be independent of the vector field \(u\). The ellipticity condition is the classical notion of ellipticity for overdetermined differential operators \citelist{\cite{Hormander1958}*{theorem 1}\cite{Spencer}*{definition 1.7.1}} (when \(\dim V = 1\), see also S.\thinspace Agmon \citelist{\cite{Agmon1959}*{\S 7}\cite{Agmon}*{definition 6.3}}):
\begin{definition}
A homogeneous linear differential operator \(A(D)\) on \(\R^n\) from \(V\) to \(E\) is \emph{elliptic} if for every \(\xi \in \R^n \setminus\{0\}\), 
\(A(\xi)\) is one-to-one.
\end{definition}

The restriction \(p > 1\) is essential in theorem~\ref{theoremEllipticLp}. Indeed, D.\thinspace Ornstein \cite{Ornstein1962} has shown that there are no nontrivial \(L^1\)-estimates of derivatives\footnote{Whereas D.\thinspace Ornstein's result does not include explicitely vector valued operators, his theorem and his method of proof remain valid in this case. B.\thinspace Kirchheim and J.\thinspace Kristensen \citelist{\cite{CKCRAS}\cite{CKPaper}} have given a proof that relies on the convexity of homogeneous rank-one convex functions; their result covers explicitely the vectorial case.}.

\begin{theorem}[D.\thinspace Ornstein, 1962]
\label{theoremOrnstein}
Let \(A(D)\) and \(B(D)\) be homogeneous linear differential operators of order \(k\) on \(\R^n\) from \(V\) to \(E\) and from \(V\) to \(\R\) respectively. If for every \(u \in C^\infty_c(\R^n; V)\), 
\[
 \norm{B(D) u}_{L^1} \le C \norm{A(D) u}_{L^1}, 
\]
then there exists \(T \in \Lin(E; \R)\) such that 
\[
 B(D) = T \circ A(D).
\]
\end{theorem}

Here \(\Lin(E; \R)\) denotes the set of linear maps from \(E\) to \(\R\). The derivatives \(B(D)u\) are then linear combinations of the derivatives \(A(D)u\) and the estimate is trivial in the sense that it follows immediately from the boundedness of linear maps defined on finite-dimensional vector spaces.

\subsection{A collection of known Sobolev inequalities and non-inequalities}
Whereas theorem~\ref{theoremEllipticLp} fails for \(p=1\), one can ask whether in some other estimates the quantity \(\norm{D^ku}_{L^1}\) can be replaced by some weaker quantity \(\norm{A(D)u}_{L^1}\).

Consider the classical Gagliardo--Nirenberg--Sobolev inequality \citelist{\cite{Gagliardo}\cite{Nirenberg1959}*{p.\thinspace 125}} which states that for every vector field \(u \in C^\infty_c(\R^n; V)\), one has
\begin{equation}
\label{ineqSobolevL1}
  \norm{u}_{L^{\frac{n}{n-1}}} \le C \norm{Du}_{L^1}.
\end{equation}
One can wonder whether all the components of the derivative \(Du\) are necessary in this estimate when \(u\) is a vector-field.

A first example of such a possibility is the Korn--Sobolev inequality of M.\thinspace J.\thinspace Strauss \cite{Strauss1973}*{theorem 1} (see also \citelist{\cite{BB2007}*{Corollary 26}\cite{VS2004ARB}*{theorem 6}}): for every \(u \in C^\infty_c(\R^n; \R^n)\), one has
\begin{equation}
\label{KornSobolevL1}
 \norm{u}_{L^{n/(n-1)}} \le C \norm{\nabla_s u}_{L^1},
\end{equation}
where \(\nabla_s u=\frac{1}{2}\bigl(Du+ (Du)^*\bigr)\) denotes the symmetric part of the derivative \(Du \in C^\infty(\R^n; \Lin(\R^n; \R^n)\). 
This inequality does not follow from \eqref{ineqSobolevL1}, as the norms \(\norm{\nabla_s u}_{L^1}\) and \(\norm{D u}_{L^1}\) are not equivalent by theorem~\ref{theoremOrnstein} (see also \cite{CFM2005}*{theorem 1}). 
In the three-dimensional space \(\R^3\), one can wonder whether an estimate of the kind
\begin{equation}
\label{NonHodgeSobolevL1R3}
  \norm{u}_{L^\frac{3}{2}} \le C\bigl(\norm{\Div u}_{L^1} + \norm{\Curl u}_{L^1} \bigr)
\end{equation}
holds for every \(u \in C^\infty_c(\R^3; \R^3)\). The answer is known to be negative even in the case where \(\Curl u=0\); a contradiction is obtained by taking suitable regularizations of the gradient of Newton's kernel \(x \in \R^3 \mapsto \frac{-x}{4\pi\abs{x}^3}\).
Surprisingly, J.\thinspace Bourgain and H.\thinspace Brezis \citelist{\cite{BB2004}*{theorem 2}\cite{BB2007}*{corollary 7}} have proved that for every vector field \(u \in C^\infty_c(\R^3; \R^3)\) such that \(\Div u=0\), one has
\begin{equation}
 \label{HodgeSobolevL1R3}
 \norm{u}_{L^\frac{3}{2}} \le C\norm{\Curl u}_{L^1}.
\end{equation}
J.\thinspace Bourgain and H.\thinspace Brezis \cite{BB2007}*{Corollary 17} have proved similarly that for every differential form \(u \in C^\infty_c(\R^n; \bigwedge^\ell \R^n)\), one has the Hodge--Sobolev inequality
\begin{equation}
 \label{HodgeSobolevL1}
 \norm{u}_{L^{n/(n-1)}} \le C \bigl(\norm{d u}_{L^1} + \norm{d^* u}_{L^1}\bigr).
\end{equation}
(see also L.\thinspace Lanzani and E.\thinspace M.\thinspace Stein \cite{LS2005}).

\subsection{Limiting Sobolev inequalities and canceling operators}
We would like to determine whether for a given first order homogeneous differential operator \(A(D)\) 
 an estimate of the form
\begin{equation}
\label{ineqA}
  \norm{u}_{L^\frac{n}{n-1}} \le C \norm{A(D)u}_{L^1}
\end{equation}
holds. The answer is given by

\begin{theorem}
\label{theoremOrderk}
Let \(A(D)\) be a homogeneous linear differential operator of order
\(k\) on \(\R^n\) from \(V\) to \(E\). The estimate
\[
  \norm{D^{k-1}u }_{L^{\frac{n}{n-1}}} \le C\norm{A(D) u}_{L^1},
\]
holds for every \(u \in C^\infty_c(\R^n; V)\)
if and only if \(A(D)\) is elliptic and canceling.
\end{theorem}

The cancellation is a new condition that we introduce

\begin{definition}
A homogeneous linear differential operator \(A(D)\) on \(\R^n\) from \(V\) to \(E\) is \emph{canceling} if 
\[
  \bigcap_{\xi \in \R^n \setminus\{0\}} A(\xi)[V]=\{0\}.
\]
\end{definition}

In the well-known \(L^p\) counterpart of theorem~\ref{theoremOrderk} for \(1 < p < n\), the ellipticity alone is sufficient. One has for every \(u \in C^\infty_c(\R^n; \R^n)\)
\begin{equation}
\label{ineqOrderkLp}
  \norm{D^{k-1}u }_{L^{\frac{np}{n-p}}} \le C\norm{A(D) u}_{L^p},
\end{equation}
if and only if \(A(D)\) is elliptic\footnote{The sufficiency of the ellipticity is a consequence of the classical theorem~\ref{theoremEllipticLp} and the Sobolev embedding. The necessity of ellipticity in \eqref{ineqOrderkLp} was probably known to the experts; we shall prove in proposition~\ref{propositionEllipticitySobolev} that ellipticity is necessary in \eqref{ineqOrderkLp} for every \(p \in [1, n)\).}.

The cancellation condition for first-order operators is equivalent to a structural condition used by J.\thinspace Bourgain and H.\thinspace Brezis to prove \eqref{ineqA} \cite{BB2007}*{theorem 25} (see proposition~\ref{propBBCanceling} below).

The sufficiency part of theorem~\ref{theoremOrderk} will be proved in proposition~\ref{propositionSufficient}; the necessity of the ellipticity in corollary~\ref{corollaryEllipticitySobolev} and the necessity of the cancellation in proposition~\ref{propositionCancelingNecessary}. 

The estimates \eqref{ineqSobolevL1}, \eqref{KornSobolevL1} and \eqref{HodgeSobolevL1} will be derived from theorem~\ref{theoremOrderk} in section~\ref{sectionExamples} as well as the nonestimate~\eqref{NonHodgeSobolevL1R3}. The case of the Hodge--Sobolev inequality \eqref{HodgeSobolevL1R3} will be treated in section~\ref{sectionPartCanceling} in a generalization of theorem~\ref{theoremOrderk} to partially canceling operators.

Theorem~\ref{theoremOrderk} also remains valid for estimates in fractional Sobolev spaces and in Lorentz spaces (section~\ref{sectionFractional}).
Using the tools of J.\thinspace Bourgain and H.\thinspace Brezis, a counterpart of theorem~\ref{theoremOrderk} with a weaker norm is obtained (section \ref{sectionBB}). 

\subsection{Estimates for $L^1$ vector fields and cocanceling operators}
By the H\"older inequality and classical elliptic estimates, the estimate
\[
 \norm{D^{k-1} u}_{L^\frac{n}{n - 1}} \le C\norm{A(D)u}_{L^1}
\]
for every \(u \in \in C^\infty_c(\R^n; V)\) is equivalent to
\[
 \int_{\R^n} A(D)u \cdot \varphi \le C' \norm{A(D)u}_{L^1} \norm{D \varphi}_{L^n}.
\]
for every \(u \in \in C^\infty_c(\R^n; V)\) and \(\varphi \in C^\infty_c(\R^n; E)\).

This leads us to the related question to determine under which conditions does one have an estimate 
\begin{equation}
\label{ineqL1W1n}
   \int_{\R^n} f \cdot \varphi \le C \norm{f}_{L^1} \norm{D \varphi}_{L^n}.
\end{equation}
for every \(\varphi \in C^\infty_c(\R^n; E)\) and every \(f\) in some subset of \(L^1(\R^n; E)\). 
Without any restriction on \(f\), this estimate fails when \(n \ge 2\); it would be indeed equivalent with 
\[
  \norm{u}_{L^\infty} \le C \norm{D u}_{L^n},
\]
which is also known to be false. 
Surprisingly, J.\thinspace Bourgain and H.\thinspace Brezis \citelist{\cite{BB2004}*{p.\thinspace  541}\cite{BB2007}*{theorem 1\('\)}} have proved that when \(E=\R^n\) and \(f\) is taken in the class of divergence-free vector-fields, the above estimate holds. We want to determine for a given differential operator \(L(D)\) on \(\R^n\) from \(E\) to \(F\), whether an estimate of the type \eqref{ineqL1W1n} holds. The answer is given by 

\begin{theorem}
\label{theoremCocanceling}
Let \(n \ge 2\) and \(L(D)\) be a homogeneous differential operator on \(\R^n\) from \(E\) to \(F\).
The following conditions are equivalent
\begin{enumerate}[(i)]
\item 
\label{itemCocancelingEstimate} there exists \(C > 0\) such that for every \(f \in L^1(\R^n; E)\) such that \(L(D)f=0\) and \(\varphi \in C^\infty_c(\R^n; E)\),
\[
  \int_{\R^n} f \cdot \varphi \le C \norm{f}_{L^1} \norm{D \varphi}_{L^n},
\]
\item 
\label{itemCocanceling0} for every \(f \in L^1(\R^n; E)\) such that \(L(D)f=0\)
\[
  \int_{\R^n} f = 0,
\]
\item \label{itemCocancelingCocanceling} \(L(D)\) is cocanceling.
\end{enumerate}
\end{theorem}

The \emph{cocancellation} condition is a new condition that we introduce:
\begin{definition}
\label{definitionCocanceling}
Let \(L(D)\) be a homogeneous linear differential operator on \(\R^n\) from \(E\) to \(F\). The operator \(L(D)\) is \emph{cocanceling} if 
\[
  \bigcap_{\xi \in \R^n \setminus \{0\}} \ker L(\xi)=\{0\}.  
\]
\end{definition}

The equivalence between \eqref{itemCocanceling0} and \eqref{itemCocancelingCocanceling} in theorem~\ref{theoremCocanceling} will be the proved in proposition~\ref{propositionEquivalentCocanceling}; \eqref{itemCocanceling0} will be deduced from \eqref{itemCocancelingEstimate} in proposition~\ref{propositionCocancelingNecessary}; \eqref{itemCocancelingEstimate} will be proved assuming \eqref{itemCocancelingCocanceling} in proposition~\ref{propCocancelingSufficient} relying on results of J.\thinspace Bourgain and H.\thinspace Brezis \cite{BB2007} and the author \cite{VS2008}.

It is possible also to obtain some partial estimate when \(L(D)\) satisfies partially the cocancellation condition (see section~\ref{sectionPartCanceling}) and to obtain fractional estimates (see section~\ref{sectionFractional}).
Using the tools of J.\thinspace Bourgain and H.\thinspace Brezis, we show that if \(L(D)\) is a canceling homogenenous differential operator, it allows to characterize vector-fields \(f \in L^1(\R^n; E)\) that define linear functionals on the homogeneous Sobolev space \(\dot{W}^{1, n}(\R^n; E)\) (see section~\ref{sectionBB}).  

\section{Estimates on $L^1$ vector fields and cocanceling operators}

\label{sectionCocanceling}

\subsection{Characterization of cocanceling operators}
The following proposition characterizes cocanceling operators:

\begin{proposition}
\label{propositionEquivalentCocanceling}
Let \(L(D)\) be a homogeneous linear differential operator of order \(k\) on \(\R^n\) from \(E\) to \(F\).
The following are equivalent
\begin{enumerate}[(i)]
 \item\label{itCocanceling} \(L(D)\) is cocanceling, 
 \item\label{itVanishDirac} for every \(e \in E\), if \(L(D)\, (\delta_0 e) =0\), then \(e = 0\),
 \item\label{itVanishL1} for every \(f \in L^1(\R^n; E)\), if \(L(D)f = 0\), then \[\int_{\R^n} f = 0,\] 
 \item\label{itVanishCinfty} for every \(f \in C^\infty_c(\R^n; E)\), if \(L(D)f = 0\), then
\[
\int_{\R^n} f = 0.
\] 
\end{enumerate}
\end{proposition}

Here \(\delta_0\) denotes Dirac's measure at \(0\). In \eqref{itVanishDirac} and \eqref{itVanishL1}, the differential operator \(L(D)\) is taken in the sense of distributions.

\begin{proof}
Assume that \(L(D)\) is cocanceling. Fix \(e \in E\) such that \(L(D)(\delta_0 e)=0\). For every \(\varphi \in C^\infty_c(\R^n; E)\), by definition of the distributional derivative and properties of the Fourier transform \(\widehat{\varphi}\) of \(\varphi\),
\[ 
\begin{split}
  \dualprod{L(D) (\delta_0 e)}{\varphi}&=(-1)^ke \cdot \bigl(L(D)^*\varphi\bigr)(0)\\
&=\int_{\R^n} e \cdot \bigl((-2\pi i)^kL(\xi)^*[\widehat{\varphi}(\xi)]\bigr)\,d\xi\\
&=\int_{\R^n} \bigl((2\pi i)^kL(\xi)[e]\bigr) \cdot \widehat{\varphi}(\xi)\,d\xi.
\end{split}
\]
Since by hypothesis \(L(D)\,(\delta_0 e)=0\),  we have, for every \(\varphi \in C^\infty_c(\R^n; E)\),
\[
 \int_{\R^n} \bigl((2\pi i)^kL(\xi)[e]\bigl) \cdot \widehat{\varphi}(\xi)\,d\xi=0;
\]
hence for every \(\xi \in \R^n\)
\(
  L(\xi)[e]=0
\).
Since \(L(D)\) is cocanceling, we conclude that \(e=0\). We have proved that 
\eqref{itCocanceling} implies \eqref{itVanishDirac}.

Now assume that \eqref{itVanishDirac} holds and let \(f \in L^1(\R^n; E)\). If \(L(D)f=0\), define \(f_\lambda : \R^n \to E\) for \(\lambda >0\) and \(x \in \R^n\) by
\(
  f_\lambda(x) = \frac{1}{\lambda^n} f\bigl(\frac{x}{\lambda}\bigr)
\).
One has \(f_\lambda \to \delta_0 e\) in the sense of distributions as \(\lambda \to 0\),
where
\(
  e=\int_{\R^n} f.
\)
Therefore \(L(D)f_\lambda \to L(D) (\delta_0e)\) in the sense of distributions as \(\lambda \to 0\).
Since \(L(D)\) is homogeneous, \(L(D)f_\lambda=0\) and hence \(L(D)(\delta_0 e)=0\). Therefore by assumption, \(\int_{\R^n} f = e = 0\). We have proved \eqref{itVanishL1}.
It is clear that \eqref{itVanishL1} implies \eqref{itVanishCinfty}.

Finally assume that \eqref{itVanishCinfty} holds. Let \(e \in \bigcap_{\xi \in \R^n \setminus \{0\}} \ker L(\xi) \). Choose \(\psi \in C^\infty_c(\R^n)\) such that \(\int_{\R^n} \psi=1\). For every \(x \in \R^n\), 
\[
 \bigl(L(D) (\psi e)\bigr) (x)=\int_{\R^n} e^{2\pi i x \cdot \xi} (2\pi i)^k L(\xi)[e] \widehat{\psi}(\xi)\,d\xi=0.
\]
By \eqref{itVanishCinfty}, we conclude that \(e=\int_{\R^n} \psi e = 0\). We have proved that \(L(D)\) is cocanceling.
\end{proof}

In general, it is not clear whether there exists \(f \in C^\infty_c(\R^n; E)\setminus\{0\}\) such that \(L(D)f=0\). When \(L(D)\) is \emph{not cocanceling}, proposition~\ref{propositionCocancelingNecessary} shows that there exists \(f \in C^\infty_c(\R^n;  E)\setminus\{0\}\) such that \(L(D)f=0\).

\subsection{Necessity of the cocancellation}
Using a classical construction, we prove that \eqref{itemCocancelingEstimate} implies \eqref{itemCocanceling0} in theorem~\ref{theoremCocanceling} 

\begin{proposition}
\label{propositionCocancelingNecessary}
Let \(n \ge 2\) and \(f \in L^1(\R^n; E)\). If for every \(\varphi \in C^\infty_c(\R^n; E)\), 
\[
  \int_{\R^n} f \cdot \varphi \le C \norm{f}_{L^1} \norm{D\varphi}_{L^n},
\]
then 
\[
  \int_{\R^n} f = 0.
\]
\end{proposition}
\begin{proof}
Let \(\psi \in C^\infty(\R^+)\) be such that \(\psi=1\) on \( [0, 1] \), \(\psi \in [0, 1]\) on \([1, 2]\) and \(\psi=0\) on \([2, \infty)\). For \(\lambda > 0\) define \(\varphi_\lambda : \R^n \to \R\) for \(x \in \R^n\) by 
\(
 \varphi_\lambda(x) =\psi(\abs{x}^\lambda).
\)
One has for every \(x \in \R^n\), \(\lim_{\lambda \to 0} \varphi_\lambda(x)=1\) and 
\(
  \norm{D \varphi_\lambda}_{L^n}=\lambda^{1-\frac{1}{n}} \norm{D \varphi_1}_{L^n}.
\)
By Lebesgue's dominated convergence theorem and the estimate, \(\int_{\R^n} f=0\).
\end{proof}

\subsection{Estimates on $L^1$ vector fields}
We shall now prove that the cocancellation condition implies the estimate in  theorem~\ref{theoremCocanceling}
\begin{proposition}
\label{propCocancelingSufficient}
Let \(L(D)\) be a homogeneous differential operator from \(E\) to \(F\). If \(L(D)\) is cocanceling,  \(f \in L^1(\R^n; E)\) and \(L(D)f=0\) in the sense of distributions, then for every \(\varphi \in C^\infty_c(\R^n; E)\), 
\[
  \int_{\R^n} f \cdot \varphi \le C \norm{f}_{L^1}\norm{D \varphi}_{L^n}.
\]
\end{proposition}

The first ingredient of the proof of proposition~\ref{propCocancelingSufficient} is a similar result in which the vector condition is replaced by a single scalar condition. It will be shown in proposition~\ref{propositionCocancelingHigherOrder} that this is a particular case of proposition~\ref{propCocancelingSufficient}.

\begin{proposition}[Van Schaftingen, 2008 \cite{VS2008}*{theorem 4}]
\label{propVSHigherOrder}
Let \(k \in \N\) and \(f_\alpha \in L^1(\R^n)\) for \(\alpha \in \N^n\) with \(\abs{\alpha}=k\). If 
\[
 \sum_{\substack{\alpha \in \N^n\\ \abs{\alpha}=k}} \partial^\alpha f_\alpha=0
\]
in the sense of distributions, then for every \(\alpha \in \N^n\) with \(\abs{\alpha}=k\) and \(\varphi \in C^\infty_c(\R^n)\)
\[
 \int_{\R^n} f_\alpha \varphi \le C \norm{f}_{L^1} \norm{D \varphi}_{L^n}.
\]
\end{proposition}

The proof of proposition~\ref{propVSHigherOrder} relies on a slicing argument which is reminiscent of that used for the proof of the Gagliardo--Nirenberg embedding \citelist{\cite{Gagliardo}\cite{Nirenberg1959}*{pp.\thinspace 128-129}}, the Korn--Sobolev inequality \cite{Strauss1973} and which is a modification of an argument for estimates of circulation along closed curves \cite{VS2004BBM}, divergence-free vector fields \cite{VS2004Divf}, closed differential forms \cite{LS2005} and vector fields that satisfy a second-order condition \cite{VS2004ARB}.
This was adapted to fractional spaces \citelist{\cite{BB2004}*{remark 1}\cite{BB2007}*{remark 11}\cite{VS2004Divf}*{remark~5}\cite{VS2006BMO}*{remark~4.2}\cite{VS2008}*{remark 2}\cite{VS2010}} and noncommutative settings \citelist{\cite{CVS2009}\cite{Yung2010}}.
A stronger version of Proposition~\ref{propVSHigherOrder} can also be obtained by the methods of J.\thinspace Bourgain and H.\thinspace Brezis \cite{BB2007} (see theorem~\ref{theoremVectorL1BB}).

The second ingredient is an algebraic lemma:
\begin{lemma}
\label{lemmaCharKalpha}
Let \(L(D)=\sum_{\abs{\alpha}=k} \partial_\alpha L_\alpha\) be a homogeneous differential operator of order \(k\) on \(\R^n\) from \(E\) to \(F\). The operator \(L(D)\) is cocanceling if and only if there exist \(K_\alpha \in \Lin(F; E)\) for every \(\alpha \in \N^n\) with \(\abs{\alpha}=k\) such that 
\begin{equation}
\label{eqCharKalpha}
 \sum_{\substack{\alpha \in \N^n\\ \abs{\alpha}=k}}K_\alpha \circ L_\alpha=\id.
\end{equation}
\end{lemma}

A key consequence of lemma~\ref{lemmaCharKalpha} is that given \(f \in L^1(\R^n; E)\) such that  \(L(D)f=0\), \(f\) is the composition of a linear map with a vector field that satisfies the assumptions of proposition~\ref{propVSHigherOrder}. Indeed by taking \(g_\alpha=L_\alpha(f)\), one can write \(f=\sum_{\alpha \in \N^n, \abs{\alpha}=k} K_\alpha (g_\alpha)\) with \(\sum_{\alpha \in \N^n, \abs{\alpha}=k} \partial^\alpha g_\alpha=0\).

\begin{proof}[Proof of lemma~\ref{lemmaCharKalpha}]
Since \((\xi_{\alpha})_{\abs{\alpha}=k}\) is a basis of the vector space of homogeneous polynomials of degree \(k\), the operator \(e \in E \mapsto \bigl(L_\alpha(e)\bigr)_{\abs{\alpha}=k} \in F^{\binom{n+k-1}{k}}\) is one-to-one if and only if \(L(D)\) is cocanceling. This is equivalent with this map being invertible on the left, which is  \eqref{eqCharKalpha}. 
\end{proof}

Proposition~\ref{propCocancelingSufficient} will now be a consequence of proposition~\ref{propVSHigherOrder} and lemma~\ref{lemmaCharKalpha}.

\begin{proof}[Proof of proposition~\ref{propCocancelingSufficient}]
By assumption \(\sum_{\abs{\alpha}=k} \partial^\alpha L_\alpha(f) = 0\). 
By proposition~\ref{propVSHigherOrder}, for every \(\alpha \in \N^n\) with \(\abs{\alpha}=k\) and \(\psi \in C^\infty_c(\R^n; V)\), 
\begin{equation}
\label{ineqLalpha}
   \int_{\R^n} L_\alpha(f)\cdot \psi \le C \norm{f}_{L^1} \norm{D \psi}_{L^n}.
\end{equation}
For \(\varphi \in C^\infty_c(\R^n; E)\), in view of \eqref{eqCharKalpha} and \eqref{ineqLalpha}
\[
 \begin{split}
\int_{\R^n} f \cdot \varphi&=\sum_{\substack{\alpha \in \N^n\\ \abs{\alpha} = k}} \int_{\R^n} L_\alpha(f) \cdot K_\alpha{}^*(\varphi)\\
& \le C \sum_{\substack{\alpha \in \N^n\\ \abs{\alpha}=k}} \norm{L_\alpha(f)}_{L^1} \norm{D K_\alpha{}^*(\varphi)}_{L^n}
  \le C' \norm{f}_{L^1} \norm{D \varphi}_{L^n}.\qedhere
\end{split}
\]
\end{proof}

\section{Examples of cocanceling operators}
\label{sectionCocancelingExamples}

\subsection{Divergence}
A first example of cocanceling operator is the divergence operator.

\begin{proposition}
\label{propositionCocancelingDivergence}
Let \(L(D)\) be the homogeneous linear differential operator of order \(1\) on \(\R^n\) from \(\R^n\) to \(\R\) defined for \(\xi \in \R^n\) and \(e \in \R^n\) by 
\[
  L(\xi)[e]=\xi \cdot e.
\]
The operator \(L(D)\) is cocanceling.
\end{proposition}
\begin{proof}
For every \( \xi \in \R^n\), \(\ker L(\xi)=\xi^\perp\). Hence, \(\bigcap_{\xi \in \R^n \setminus \{0\}}=\{0\}\).
\end{proof}

As a consequence of theorem~\ref{theoremCocanceling}, we recover the estimate
\begin{corollary}[J.\thinspace Bourgain and H.\thinspace Brezis, 2004 \citelist{\cite{BB2004}*{p. 541}\cite{BB2007}*{theorem \(1^\prime\)}\cite{VS2004Divf}*{theorem 1.5}}]
For every \(f \in L^1(\R^n; \R^n)\) such that \(\Div f = 0\) and every \(\varphi \in C^\infty_c(\R^n)\), 
\[
  \int_{\R^n} f \cdot \varphi \le C \norm{f}_{L^1} \norm{D \varphi}_{L^n}.
\]
\end{corollary}

\subsection{Exterior derivative}
The construction for the divergence operator generalizes to differentials forms

\begin{proposition}\label{propositionCocancelingd}
Let \(\ell \in \{0, \dotsc, n-1\}\) and \(L(D)\) be the homogeneous linear differential operator of order \(1\) on \(\R^n\) from \(\bigwedge^\ell \R^n\) to \(\bigwedge^{\ell+1}\R^n\) defined for \(\xi \in \R^n \simeq \bigwedge^1 \R^n\) and \(e \in \bigwedge^\ell \R^n\) by 
\[
  L(\xi)[e]=\xi \wedge e.
\]
The operator \(L(D)\) is cocanceling.
\end{proposition}
\begin{proof}
If \(e \in \bigwedge^\ell \R^n\) with \(\ell \le n-1\), one checks that if \(\xi \wedge e=0\) for every \(\xi \in \R^n\), then \(e=0\).
\end{proof}

As a consequence we recover from theorem~\ref{theoremCocanceling} the estimate
\begin{corollary}[J.\thinspace Bourgain and H.\thinspace Brezis \cite{BB2007}*{Corollary 17}, 2004 and L. \thinspace Lanzani and E.\thinspace Stein, 2005 \cite{LS2005}]
Let \(\ell \in \{0, \dotsc, n-1\}\). For every \(f \in L^1(\R^n; \bigwedge^\ell \R^n)\) such that \(d f = 0\) and every \(\varphi \in C^\infty_c(\R^n; \bigwedge^{n-\ell} \R^n)\), 
\[
  \int_{\R^n} f \wedge \varphi \le C \norm{f}_{L^1} \norm{D \varphi}_{L^n}.
\]
\end{corollary}

\subsection{Higher order condition}
One can also replace the divergence with a similar higher-order condition

\begin{proposition}
\label{propositionCocancelingHigherOrder}
Let \(k \in \N_*\) and \(L(D)\) be the homogeneous linear differential operator of order \(k\) on \(\R^n\) from \(\R^{\binom{n+k-1}{k}}\) to \(\R\) defined for \(\xi \in \R^n\) and \(e \in  \R^{\binom{n+k-1}{k}}\) by 
\[
  L(\xi)[e]=\sum_{\substack{\alpha \in \N^n\\\abs{\alpha}=k}} \xi^\alpha e_\alpha.
\]
The operator \(L(D)\) is cocanceling.
\end{proposition}
\begin{proof}
Assume that \(e \in \bigcap_{\xi \in \R^n \setminus \{0\}} \ker L (\xi)\). One has then for every \(\xi \in \R^n\), \(\sum_{\alpha \in \N^n, \abs{\alpha}=k} \xi^\alpha e_\alpha=0\). By the properties of multivariate polynomials, one concludes that \(e=0\).
\end{proof}

As a corollary, one recovers proposition~\ref{propVSHigherOrder} from theorem~\ref{theoremCocanceling}.

\subsection{Saint-Venant compatibility conditions}
The Saint-Venant compatibility conditions are an example of cocanceling operator. In order to define it, denote by \(S^2\R^n\) the space of symmetric bilinear forms on \(\R^n\).

\begin{proposition}
\label{propositionSaintVenantCocanceling}
Let  \(W(D)\) be the homogeneous linear differential operator of order \(2\) on \(\R^n\) from \(S^2\R^n\) to \(S^2\R^n \otimes S^2\R^n\) defined for \(\xi \in \R^n\), \(e \in S^2 \R^n\),  and \(u, v, w, z \in \R^n\) by
\begin{multline*}
  \bigl(W(\xi)[e]\bigr)[u, v, w, z]=e(u,v) (\xi \cdot w) (\xi \cdot z)+e(w,z) (\xi \cdot u) (\xi \cdot v) \\
-e(u, z) (\xi \cdot w) (\xi \cdot v)-e(w, v) (\xi \cdot u) (\xi \cdot z).
\end{multline*}
The operator \(W(D)\) is cocanceling if and only if \(n\ge 2\).
\end{proposition}

\begin{proof}
First note that if \(n=1\), \(L(D)=0\). 

Assume that \(n \ge 2\) and let \(e \in S^2 \R^n\) be such that for every \(u, v, w, z \in \R^n\) and \(\xi \in \R^n\),
\begin{equation}
\label{eqSaintVenantW}
  \bigl(W(\xi)[e]\bigr)[u, v, w, z]=0.
\end{equation}
Let \(u \in \R^n\). Since \(n \ge 2\), one can choose \(w \in \R^n \setminus\{0\}\) such that \(w \cdot u = 0\). One has then
\[
  \bigl(W(w)[e]\bigr)[u, u, w, w]=e(u, u) \norm{w}^4,
\]
from which one deduces by \eqref{eqSaintVenantW} that for every \(u \in \R^n\), \(e(u, u)=0\). Since \(e\) is symmetric, \(e=0\). 
\end{proof}

\begin{corollary}
\label{corollarySaintVenant}
Let \(n \ge 2\). For every \(f \in L^1(\R^n; S^2 \R^n)\) such that \(W(D) f = 0\) and every \(\varphi \in C^\infty_c(\R^n; S^2 \R^n)\), 
\[
  \int_{\R^n} f : \varphi \le C \norm{f}_{L^1} \norm{D \varphi}_{L^n}.
\]
\end{corollary}

Here \( : \) denotes the scalar product in \(S^2\R^n\). Corollary~\ref{corollarySaintVenant} is the core of the argument of the proof of the Korn--Sobolev inequality by estimates under second order conditions \cite{VS2004ARB}*{theorem 6}.

\smallbreak

We can also consider higher-order Saint-Venant operators  \cite{Sharafutdinov}*{(2.1.9)}. We denote by \(S^k\R^n\) the space of symmetric \(k\)-linear forms on \(\R^n\).

\begin{proposition}
\label{propositionSaintVenantCocancelingHigherOrder}
Let \(W(D)\) be the homogeneous linear differential operator of order \(k\) on \(\R^n\) from \(S^k\R^n\) to \(S^k\R^n \otimes S^k\R^n\) defined for \(\xi \in \R^n\), \(e \in S^k \R^n\),  and \(v^0_1, \dotsc, v^0_k, v^1_1, \dotsc, v^1_k \in \R^n\) by
\begin{multline*}
 \bigl(W(\xi)[e]\bigr)[v^0_1, \dotsc, v^0_k, v^1_1, \dotsc, v^1_k]\\
=\sum_{\alpha \in \{0, 1\}^n}(-1)^{\abs{\alpha}} e(v_1^{\alpha_1}, \dotsc, v_k^{\alpha_k})(\xi \cdot v_1^{1-\alpha_1}) \dotsm (\xi \cdot v_k^{1-\alpha_k}).
\end{multline*}
The operator \(W(D)\) is cocanceling if and only if \(n \ge 2\), .
\end{proposition}

The condition \(W(D)f=0\) is satisfied by the symmetric derivative of a field of symmetric \(k-1\)-linear forms.

\begin{proof}[Sketch of the proof of proposition~\ref{propositionSaintVenantCocancelingHigherOrder}]
Assume that \(e \in \bigcap_{\xi \in \R^n \setminus \{0\}} \ker L(\xi)\). Given \(u \in \R^n\), one chooses \(w \in \R^n \setminus \{0\} \) such that \(w \cdot u = 0\). One has then 
\[
  \bigl(W(w)[e]\bigr)[u, \dotsc, u, w,\dotsc, w]=e(u, \dotsc, u) \norm{w}^{2k},
\]
from which one concludes that \(e=0\).
\end{proof}

\section{Proof of the Sobolev estimate}

\label{sectionSobolev}

In this section we prove a Sobolev estimate for elliptic canceling operator. We proceed in several steps. First we recall in section~\ref{subsectionClassicalElliptic} a classical elliptic estimate for elliptic operators. Next in section~\ref{subsectionCompatibility} we  recall how the range of a given linear differential operator can be characterized as the kernel of another linear differential operator of compatibility conditions and we study when this operator is cocanceling. Finally, in section~\ref{subsectionSobolev}, we prove the estimate by combining the previous ingredients with theorem~\ref{theoremCocanceling} proved in section~\ref{sectionCocanceling}. 

\subsection{Classical elliptic estimates}
\label{subsectionClassicalElliptic}
In order to prove theorem~\ref{theoremOrderk}, we shall use a classical variant of theorem~\ref{theoremEllipticLp}

\begin{proposition}
\label{propCaldZygmundWk1p}
Let \(A(D)\) be a linear homogeneous differential operator of order \(k\) on \(\R^n\) from \(V\) to \(E\).
If \(A(D)\) is elliptic and \(p > 1\), then for every \(u \in C^\infty_c(\R^n; V)\), 
\[
  \norm{D^{k-1} u}_{L^p} \le C \norm{A(D) u}_{\dot{W}^{-1, p}}.
\]
\end{proposition}
\begin{proof}
One has for every  \(\alpha \in \N^n\) with \(\abs{\alpha}=k-1\) and for every \(\xi \in \R^n \setminus \{0\}\),  
\[
  \widehat{\partial^\alpha u}(\xi)=\frac{1}{2\pi i} \xi^{\alpha}  \bigl(A(\xi)^* \circ A(\xi)\bigr)^{-1}\circ A(\xi)^*\bigl(\widehat{A(D)u}(\xi)\bigr).
\]
Recall that \(\norm{A(D) u}_{\dot{W}^{-1, p}}=\norm{(-\Delta)^{-\frac{1}{2}} A(D) u}_{L^p}\).
By the theory of singular integrals on \(L^p\) (see for example E.\thinspace Stein \cite{Stein1970SIDPF}*{theorem 6 in Chapter 3, \S{} 3.5 together with theorem 3 in Chapter 2, \S{} 4.2}), one has the desired estimate. 
\end{proof}

In general \(A(D)\) is an overdetermined elliptic operator; as a consequence, there are many possible choices for a singular integral operator that inverts \(A(D)\). In the proof of proposition~\ref{propCaldZygmundWk1p}, a change of the Euclidean structure on \(E\) would result in a different singular integral operator that would have the same properties.

\subsection{Compatibility conditions}
\label{subsectionCompatibility}

The last tool in the proof of the sufficiency part in theorem~\ref{theoremOrderk} is 

\begin{proposition}
\label{propositionConstructionL}
Let \(A(D)\) be a homogeneous differential operator on \(\R^n\) from \(V\) to \(E\). If \(A(D)\) is elliptic, then there exists a finite-dimensional vector space \(F\) and a homogeneous differential operator \(L(D)\) on \(\R^n\) from \(E\) to \(F\) such that for every \( \xi \in \R^n \setminus \{0\}\), 
\[
 \ker L(\xi) = A(\xi)[V].
\]
\end{proposition}

In the language of homological algebra, for every \( \xi \in \R^n \setminus \{0\}\), 
\[ 
  V \xrightarrow{A(\xi)} E  \xrightarrow{L(\xi)} F
\]
forms an exact sequence.

The proof will be done in two steps. First we will recall the construction due to L.\thinspace Ehrenpreis \citelist{\cite{Ehrenpreis}\cite{Komatsu}*{theorem 2}\cite{Spencer}*{theorem 1.5.5}} of compatibility condition for an overdetermined linear differential operator that does not need to be elliptic. We then show that under the ellipticity condition, this operator has the required property.

Let \(\mathcal{P}^\ell_\xi(\R^n; V)\) be the space of exponential polynomials of degree at most \(\ell\), that is the set of functions \(u : \R^n \to V\) that can be written for every \(x \in \R^n\) as
\[
  u(x)=\sum_{\substack{\alpha \in \N^n \\ \abs{\alpha}\le \ell}} x^\alpha e^{\xi \cdot x} v_\alpha.
\]
where \(v_\alpha \in V\) for each \(\alpha \in \N^n\) with \(\abs{\alpha}\le \ell\).
We also set \(\mathcal{P}_\xi(\R^n; V)=\bigcup_{\ell \in \N} \mathcal{P}^\ell_\xi(\R^n; V)\).
If we define for \(\xi \in \R^n\) the function \(e_\xi : \R^n \to \R\) by \(e_\xi(x)=e^{\xi \cdot x}\) for every \(x \in \R^n\), one has \( \mathcal{P}^\ell_\xi(\R^n; V)=e_\xi \mathcal{P}^\ell_0(\R^n; V)\).

Finally, \(K(D)\) is a linear differential operator on \(\R^n\) from \(E\) to \(F\) of order at most \(\ell\) if it can be written for \(u \in C^\infty\) as 
\(
  K(D)u=\sum_{\alpha \in \N^n, \abs{\alpha} \le \ell} K_\alpha(\partial^\alpha u).
\)

The next lemma gives a necessary and sufficient condition for the solvability of the equation \(A(D)u=f\) in the framework of exponential polynomials.

\begin{lemma}
\label{lemmaExponentialPolynomials}
Let \(A(D)\) be a linear differential operator of order at most \(k\) on \(\R^n\) from \(V\) to \(E\) and let \(\xi \in \R^n\).
For every \(f \in \mathcal{P}^\ell_\xi(\R^n; E)\), there exists \(u \in \mathcal{P}^{\ell+k}_\xi(\R^n; V)\) such that \(A(D)u=f\) if and only if for every linear differential operator \(K(D)\) on \(\R^n\) of order at most \(\ell\) from \(E\) to \(\R\) such that \(K(D) \circ A(D)=0\), one has \(K(D)f=0\).
\end{lemma}

\begin{proof}
Note that for every linear form \(\phi\) on \(\mathcal{P}^{\ell}_\xi(\R^n; E)\) there exists a unique differential operator \(K(D)\) of order at most \(\ell\) on \(\R^n\) from \(E\) to \(\R\) such that for every \(g \in \mathcal{P}^{\ell}_\xi(\R^n; E)\), \(\langle \phi, g\rangle = (K(D)g)(0)\). 
If we want to characterize \(A(D) \mathcal{P}^{\ell+k}_\xi(\R^n; V)\) by duality, we
are led to study the differential operators \(K(D)\) of order at most \(\ell\) on \(\R^n\) from \(E\) to \(\R\) such that \(K(D) \circ A(D) u (0)=0\) for every \(u \in \mathcal{P}^{\ell+k}_\xi(\R^n; V)\). Note that since \(K(D)\circ A(D)\) is of order at most \(k+\ell\), this is equivalent with \(K(D) \circ A(D)=0\), which is the condition appearing in the proposition.
\end{proof}

The drawback of the previous lemma is that the number of conditions imposed on the data \(f\) depends on the degree of \(f\). This can be improved by some commutative algebra construction.

\begin{lemma}
\label{lemmaConstructionLHilbert}
Let \(A(D)\) be a linear differential operator of order \(k\) on \(\R^n\) from \(V\) to \(E\). There exists a finite dimensional vector space \(G\) and a linear differential operator \(J(D)\) from \(E\) to \(G\) such that for every \(f \in \mathcal{P}_\xi(\R^n; E)\), there exists \(u \in \mathcal{P}_\xi(\R^n; V)\) such that \(A(D)u=f\) if and only if \(J(D)f=0\).
\end{lemma}

In the language of homological algebra, the sequence
\begin{equation}
\label{eqsequence}
  \mathcal{P}_\xi(\R^n; V) \xrightarrow{A(D)} \mathcal{P}_\xi(\R^n; E) \xrightarrow{J(D)} \mathcal{P}_\xi(\R^n; G)
\end{equation}
is exact.

\begin{proof}[Proof of lemma~\ref{lemmaConstructionLHilbert}]
Let \(\mathcal{K}\) be the set of linear differential operators \(K(D)\) on \(\R^n\) from \(E\) to \(\R\) such that \(K(D) \circ A(D)=0\). The set \(\mathcal{K}\) is a submodule of the module of linear differential operators on \(\R^n\) from \(V\) to \(\R\) on the ring of linear differential operators on \(\R^n\) from \(\R\) to \(\R\) which is isomorphic to the ring of polynomials on \(\R^n\). Therefore, \(\mathcal{K}\) is finitely generated (see for example \cite{BW}*{proposition 3.32 and corollary 4.7}): there exists a finite-dimensional space \(G\) and a linear differential operator \(J(D)\) on \(\R^n\) from \(E\) to \(G\) such that for every \(K(D) \in \mathcal{K}\), there exists a differential operator \(Q(D)\) from \(G\) to \(\R\) such that \(K(D)=Q(D)\circ J(D)\). The lemma then follows from the application of lemma~\ref{lemmaExponentialPolynomials}.
\end{proof}

One can ensure that \(J(D)\) has minimal order by using tools of computational commutative algebra \cite{BW}*{\S 6.1 and 10.3}.

\smallbreak

In order to complete the proof of proposition~\ref{propositionConstructionL}, we need to show that for every \(\xi \in \R^n \setminus \{0\}\), \(\ker J(\xi)=A(\xi)[V]\). This is equivalent to the exactness of the sequence
\begin{equation}
\label{eqsequence0}
  \mathcal{P}^{0}_\xi(\R^n; V) \xrightarrow{A(D)} \mathcal{P}^0_\xi(\R^n; E) \xrightarrow{J(D)} \mathcal{P}^{0}_\xi(\R^n; G).
\end{equation}
Under the ellipticity condition, the exactness of the sequence \eqref{eqsequence} implies the exactness of the sequence \eqref{eqsequence0}:

\begin{lemma}
\label{lemmaEllipticDegree}
Let \(A(D)\) be a homogeneous linear differential operator of order \(k\) on \(\R^n\) from \(V\) to \(E\),  \(\xi \in \R^n \setminus\{0\}\), \(\ell \in \N\) and \(u \in \mathcal{P}_\xi(\R^n; V)\). If the operator \(A(D)\) is elliptic and \(A(D)u \in \mathcal{P}^\ell_\xi(\R^n; E)\), then \(u \in \mathcal{P}^\ell_\xi(\R^n; V)\).
\end{lemma}

 The lemma implies that if \(A(D)\) is elliptic, \(\ell \in \N\) and \(\xi \in \R^n \setminus \{0\}\), the sequence
\[
   \mathcal{P}^{\ell}_\xi(\R^n; V) \xrightarrow{A(D)} \mathcal{P}^\ell_\xi(\R^n; E) \xrightarrow{J(D)} \mathcal{P}^\ell_\xi(\R^n; G)
\]
is exact. When \(\ell = 0\), this is \eqref{eqsequence0}.

\begin{proof}[Proof of lemma~\ref{lemmaEllipticDegree}]
It is sufficient to show that if \(u \in \mathcal{P}^{\ell+1}_\xi(\R^n; V)\) and \(A(D)u \in \mathcal{P}^\ell_\xi(\R^n; E)\), then \(u \in \mathcal{P}^{\ell}_\xi(\R^n; V)\).
Write \(u=e_\xi p\), with \(p \in \mathcal{P}^{\ell+1}_0(\R^n; V)\). One has
\[
 A(D)[e_\xi p]=e_\xi (A(D+\xi) p)=e_\xi \bigl(A(\xi)[p]+\bigl(A(\xi+D)-A(\xi)\bigr)[p]\bigr).
\]
Note that \(\bigl(A(\xi+D)-A(\xi) \bigr)[p] \in \mathcal{P}^{\ell}_\xi(\R^n)\). Therefore, 
\(
 A(\xi)[p]\in \mathcal{P}^\ell_0(\R^n; E).
\)
Since \(A(\xi)\) is one-to-one, this implies that \(p \in \mathcal{P}^\ell_0(\R^n; V)\).
\end{proof}

\begin{proof}[Proof of proposition~\ref{propositionConstructionL}]
Let \(J(D)\) be given by lemma~\ref{lemmaConstructionLHilbert}. 
In view of lemma~\ref{lemmaConstructionLHilbert} and lemma~\ref{lemmaEllipticDegree}, one has for every \(\xi \in \R^n \setminus \{0\}\), 
\[
\begin{split}
 \ker J(\xi) &\simeq \bigl\{ f  \in \mathcal{P}^0_\xi(\R^n; E) \st J(D)f=0 \bigr\}\\
&=\bigl\{ A(D)u \st u \in \mathcal{P}^0_\xi(\R^n; V) \bigr\} \simeq A(\xi)[V],
\end{split}
\]
where the isomorphism is given by \(e \in E  \mapsto e_\xi e \in \mathcal{P}^0_\xi(\R^n; E)\).

There exist \(\nu \in \N\) and for every \(i \in \{0, \dotsc, \nu\}\) homogeneous differential operators \(J_i(D)\) of order \(i\) on \(\R^n\) from \(E\) to \(G\) such that 
\[
  J(D)=\sum_{i=0}^\nu J_i(D).
\]
Since \(A(D)\) is homogeneous, one has for every \(\xi \in \R^n \setminus \{0\}\), 
\[
  \bigcap_{i=0}^\nu \ker J_i(\xi)=A(\xi)[V].
\]
Therefore, by taking \(F=\prod_{i=0}^\nu \bigl(\bigotimes^{\nu-i} \R^n\bigr) \otimes G \) and \[L(\xi)=\bigl(\xi^{\otimes \nu} \otimes J_0(\xi), \xi^{\otimes \nu-1}\otimes J_1(\xi), \dotsc, \xi \otimes J_{\nu-1}(\xi), J_\nu(\xi)\bigr),	\] we obtain a homogeneous differential operator that has the required properties.
\end{proof}

The ellipticity assumption in proposition~\ref{propositionConstructionL} might seen unnatural in the statement. It is nevertheless essential as shown by the following example
\begin{example}
\label{exampleEllipticityCompatibility}
Consider the homogeneous linear differential operator \(A(D)\) of order \(1\) on \(\R^2\) from \(\R^2\) to \(\R^2\) defined by the matrix.
\[
 A(\xi)=\begin{pmatrix}
          \xi_1 & -\xi_2\\
          \xi_2 & -\xi_1\\
        \end{pmatrix}.
\]
The operator \(A(D)\) is not elliptic, since \(A(\xi)\) is not one-to-one when \(\abs{\xi_1}=\abs{\xi_2}\). (The reader will note that this is a hyperbolic operator.)
Assume now that there exists a homogeneous differential operator \(L(D)\) from \(\R^2\) to a vector space \(F\) such that for every \(\xi \in \R^2\), 
\[
  L(\xi)\circ A(\xi)=0.
\]
Since \(A(\xi)\) is onto when \(\abs{\xi_1} \ne \abs{\xi_2}\), we have \(L(\xi)=0\) when \(\abs{\xi_1}\ne \abs{\xi_2}\). From this we conclude that \(L(\xi)=0\) for every \(\xi \in \R^2\).
One has then \(\ker L(1,1)=\R^2 \ne \R (1, 1) = A(1,1)[\R^2]\).
Also note that since \(A(1, 1)[\R^2]=\R (1, 1)\) and 
\(A(1, -1)[\R^2]=\R (1, -1)\), \(A(D)\) is canceling, but \(L(D)=0\) is not cocanceling.
\end{example}

\begin{remark}
It is also possible to obtain an operator \(L(D)\) satisfying the conclusion of proposition~\ref{propositionConstructionL} by setting
\begin{equation}
\label{eqDefinitionL}
 L(\xi)= \det \bigl(A(\xi)^* \circ A(\xi)\bigr)\id - A(\xi) \circ \adj \bigl(A(\xi)^* \circ A(\xi)  \bigr) \circ A(\xi)^*,
\end{equation}
where \(\adj\bigl( A(\xi)^* \circ A(\xi)\bigr)=\det\bigl( A(\xi)^* \circ A(\xi)\bigr) \bigl(A(\xi)^* \circ A(\xi)\bigr)^{-1}\) is the adjugate operator of \(A(\xi)^* \circ A(\xi)\).
(This construction is up to the multiplicative constant \(\det \bigl(A(\xi)^* \circ A(\xi)\bigr)\) the classical orthogonal projector on \(A(\xi)[V]\) used for example for least-square solutions of overdetermined systems.)
The latter construction of \(L(D)\) can be much more complicated that necessary.
For example, if one is interested in the Hodge--Sobolev inequality \eqref{HodgeSobolevL1}, one takes \(V= \bigwedge^\ell \R^n\) and for every \(u \in C^\infty_c(\R^n; V)\), 
\[
  A(u)=(du, d^*u).
\]
The operator \(L(D)\) given by \eqref{eqDefinitionL} is
\[
 L(g, h)=\bigl((-\Delta)^{m-1}  d^*dg, (-\Delta)^{m-1}d d^*h\bigr),
\]
where \(m=\dim \bigwedge^{\ell+1} \R^n+\dim \bigwedge^{\ell-1} \R^n=\binom{n}{\ell-1}+\binom{n}{\ell+1}\). It is possible to show that \(L(g, h)=0\) if and only if \(dg=0\) and \(d^*h=0\).
\end{remark}

\subsection{Sobolev inequality}
\label{subsectionSobolev}
We now have all the ingredients to prove the sufficiency part of theorem~\ref{theoremOrderk}

\begin{proposition}
\label{propositionSufficient}
Let \(A(D) \) be a linear differential operator of order \(k\) on \(\R^n\) from \(V\) to \(E\).
If \(A(D)\) is elliptic and canceling, then for every \(u \in C^\infty_c(\R^n; V)\), 
\[
  \norm{D^{k-1} u}_{L^{\frac{n}{n-1}}} \le C\norm{A(D)u}_{L^1}.
\]
\end{proposition}

\begin{proof}
Let \(L(D)\) be given by proposition~\ref{propositionConstructionL}. One notes that 
\[
 L(D) \bigl(A(D)u\bigr) = 0.
\]
Since \(A(D)\) is canceling and for every \(\xi \in \R^n \setminus \{0\}\), \(\ker L(\xi) = A(\xi)[V]\), \(L(D)\) is cocanceling. Therefore, by theorem~\ref{theoremCocanceling},
\[
 \norm{A(D)u}_{\dot{W}^{-1, n/(n-1)}} \le C \norm{A(D)u}_{L^1}.
\]
Finally, we note that by proposition~\ref{propCaldZygmundWk1p}, one has
\[
 \norm{D^{k-1} u}_{L^{\frac{n}{n-1}}} \le C \norm{A(D)u}_{\dot{W}^{-1, n/(n-1)}}.\qedhere
\]
\end{proof}

\section{Necessary conditions for the Sobolev estimate}

In the section, we study the necessity of the ellipticity (section~\ref{sectionNecessaryEllipticity}) and cancellation (section~\ref{sectionNecessaryCancellation}) conditions for the Sobolev estimate. 

\subsection{Necessity of the ellipticity}
\label{sectionNecessaryEllipticity}

We show that the ellipticity condition is necessary in Sobolev-type inequalities

\begin{proposition}
\label{propositionEllipticitySobolev}
Let \(A(D)\) be a homogeneous linear differential operator of order \(k\) on \(\R^n\) from \(V\) to \(E\), \(B(D)\) be a homogeneous differential operator of order \(k-1\) on \(\R^n\) from \(V\) to \(F\), and \(p \in [1, n)\).
If for every \(u \in C^\infty_c(\R^n; E)\),
\[
 \norm{B(D)u}_{L^{\frac{np}{n-p}}}\le C \norm{A(D)u}_{L^{p}},
\]
then for every \(\xi \in \R^n\), \(\ker A(\xi) \subseteq \ker B(\xi) \).
\end{proposition}

As a corollary, we have the necessity of the ellipticity in theorem~\ref{theoremOrderk}:

\begin{corollary}
\label{corollaryEllipticitySobolev}
Let \(A(D)\) be a homogeneous linear differential operator of order \(k\) on \(\R^n\)  from \(V\) to \(E\).
If for every \(u \in C^\infty_c(\R^n; V)\), 
\[
 \norm{D^{k-1} u}_{L^{\frac{np}{n-p}}} \le C \norm{A(D) u}_{L^p},
\]
then \(A(D)\) is elliptic.
\end{corollary}
\begin{proof}
Take \(B(D)=D^{k-1}\). For every \(\xi \in \R^n \setminus \{0\}\), one has \(\ker B(\xi)=\{0\}\). The conclusion follows from the application of proposition~\ref{propositionEllipticitySobolev}.
\end{proof}

\begin{proof}[Proof of proposition~\ref{propositionEllipticitySobolev}]
Let \(\xi \in \R^n \setminus \{0\}\) and \(v \in \ker A(\xi)\).
Choose \(\varphi \in C^\infty(\R)\setminus \{0\}\) such that \(\supp \varphi \subset (-1, 1)\) and \(\psi \in C^\infty_c(\R^n)\) such that \(\psi \not \equiv 0\) on the hyperplane \(H=\{ x \in \R^n \st \xi \cdot x = 0\}\).
For \( \lambda > 0 \), define \(u_\lambda : \R^n \to \R\) for \(x \in \R^n\) by
\[
 u_\lambda(x)=\varphi( \xi \cdot x) \psi\bigl(\tfrac{x}{\lambda}\bigr)v.
\]
Since \(A(\xi)[v]=0\), one has for each \(x \in \R^n\) and \(\lambda > 0\)
\[
 \abs{A(D)u_\lambda(x)}\le C\sum_{i=1}^k \lambda^{-i} \bigabs{D^i \psi\bigl(\tfrac{x}{\lambda}\bigr)}.
\]
One has therefore, for every \(\lambda > 0\), 
\[
  \int_{\R^n} \abs{A(D) u_\lambda}^p \le C\sum_{i=1}^k \int_{H_\lambda} \lambda^{n-ip}\abs{D^i \psi}^p,
\]
where \(H_\lambda = \{x \in \R^n \st \abs{\xi \cdot x} \le \lambda^{-1} \}\).
Since for every \(i \in \{1, \dotsc, k\}\), 
\[
 \lim_{\lambda \to \infty} \lambda 
\int_{H_\lambda} \abs{D^i \psi}^p 
=\int_{H} \abs{D^i\psi}^p
\]
we conclude that, as \(\lambda \to \infty\), 
\[
  \int_{\R^n} \abs{A(D) u_\lambda}^p=O\bigl(\lambda^{n-1-p}\bigr).
\]

On the other hand, for every \(x \in\R^n\) and \(\lambda > 0\), 
\[
 \bigabs{B(D)u_\lambda(x)-\psi \bigl(\tfrac{x}{\lambda}\bigr)\varphi^{(k-1)}(\xi \cdot x) B(\xi)[v]}\le C \sum_{i=1}^{k-1} \lambda^{-i} \bigabs{D^i \psi\bigl(\tfrac{x}{\lambda}\bigr)}.
\]
As previously, we have as \(\lambda \to \infty\),
\[
 \int_{\R^n} \bigabs{B(D)u_\lambda(x)-\psi \bigl(\tfrac{x}{\lambda}\bigr)\varphi^{(k-1)}(\xi \cdot x) B(\xi)[v]}^\frac{np}{n-p}\,dx = O\bigl(\lambda^{n-1-\frac{np}{n-p}}\bigr),
\]
whence
\[
 \int_{\R^n} \abs{B(D) u_\lambda}^{p^*}
= \lambda^{n-1}\abs{B(\xi)[v]}^{p^*} \int_{\R} \abs{\varphi^{(k-1)}}^{p^*} \int_{H} \abs{\psi}^{p^*}+o(\lambda^{n-1}).
\]
Therefore, in view of the assumption, we have, as \(\lambda \to \infty\), 
\[
 \bigabs{B(\xi)[v]}\lambda^{\frac{n-1}{p}-1+\frac{1}{n}} = O(\lambda^{\frac{n-1}{p}-1}).
\]
This is only possible if \(v \in \ker B(\xi)\).
\end{proof}

The proof of proposition~\ref{propositionEllipticitySobolev} strongly relies on the fact that we are considering \(k-1\)-th derivatives on the left-hand side of the estimate. For lower derivatives one can still obtain some inequality without the ellipticity of \(A(D)\).

Consider the homogeneous linear differential operator \(A(D)\) of order \(2\) on \(\R^4\) from \(\R\) to \(\R^2\) defined for \(u \in C^\infty(\R^4)\)  by
\[
 A(D)[u]=(\partial_1\partial_2 u, \partial_3 \partial_4 u).
\]
Since \(\ker A(1, 0, 1, 0)=\R\), this operator is not elliptic. By corollary~\ref{corollaryEllipticitySobolev}, there exists \(b \in \R^4\) such that the estimate
\[
 \norm{b \cdot \nabla u}_{L^{4/3}} \le C \bigl(\norm{\partial_1 \partial_2 u}_{L^1}+\norm{\partial_3\partial_4 u}_{L^1}\bigr)
\]
does not hold.
In fact, the estimate does not hold for any \(b\in \R^4 \setminus \{0\}\).

\begin{proposition}
\label{propositionStrangeOperatorNoSobolev}
Let \(b \in \R^4\). If for every \(u \in C^\infty_c(\R^4; \R)\), 
\[
 \norm{b \cdot \nabla u}_{L^{4/3}} \le C \bigl(\norm{\partial_1 \partial_2 u}_{L^1}+\norm{\partial_3\partial_4 u}_{L^1}\bigr),
\]
then \(b=0\).
\end{proposition}
\begin{proof}
By proposition~\ref{propositionEllipticitySobolev}, if \(\xi \in \R^4\) satisfies \(\xi_1\xi_2=0\) and \(\xi_3\xi_4=0\), then  \(b \cdot \xi=0\). By taking for \(\xi\) elements of the canonical basis of \(\R^4\), one concludes that \(b=0\).
\end{proof}

On the other hand

\begin{proposition}
\label{propositionStrangeOperatorSobolev}
For every \(u \in C^\infty_c(\R^4; \R)\),
\begin{equation}
\label{ineqStrangerHigherSobolev}
 \norm{u}_{L^2} \le C\bigl(\norm{\partial_1 \partial_2 u}_{L^1}+\norm{\partial_3\partial_4 u}_{L^1}\bigr).
\end{equation}
\end{proposition}

\begin{proof}
The proof is a direct adaptation of a proof of E.\thinspace Gagliardo \cite{Gagliardo}*{teorema 5.I} and L.\thinspace Nirenberg \cite{Nirenberg1959}*{128--129}. The proof goes as follows: for every \(x=(x_1, x_2, x_3, x_4) \in \R^4\)
\[
 u(x)=\int_{-\infty}^{x_1}\int_{\infty}^{x_2} \partial_1 \partial_2 u(s, t, x_3, x_4)\,dt\,ds.
\]
Hence, for every \(x \in \R^4\), 
\[
 \abs{u(x)} \le \int_{\R^2} \abs{\partial_1 \partial_2 u(s, t, x_3, x_4)}\,ds\,dt.
\]
Similarly, one has for every \(x \in \R^4\), 
\[
 \abs{u(x)} \le \int_{\R^2} \abs{\partial_3 \partial_4 u(x_1, x_2, s, t)}\,ds\,dt.
\]
Therefore, for every \(x\in \R^4\)
\[
 \abs{u(x)}^2 \le \int_{\R^2} \abs{\partial_1 \partial_2 u(s, t, x_3, x_4)}\,ds\,dt \int_{\R^2} \abs{\partial_3 \partial_4 u(x_1, x_2, s, t)}\,ds\,dt.
\]
The integration of this inequality with respect to \(x\) on \(\R^4\) and the application of Young's inequality yields \eqref{ineqStrangerHigherSobolev}.
\end{proof}

We have thus an operator which is not elliptic. By proposition~\ref{propositionStrangeOperatorNoSobolev}, there is no first-order Sobolev inequality, but there is a second-order Sobolev inequality of proposition~\ref{propositionStrangeOperatorSobolev}.

\subsection{Necessity of the cancellation} 
\label{sectionNecessaryCancellation}

The necessity of the cancellation property for Sobolev-type estimates is given by the following

\begin{proposition}
\label{propositionCancelingNecessary}
Assume that \(A(D)\) is an elliptic homogeneous linear differential operator of order \(k\) on \(\R^n\) from \(V\) to \(E\). Let \(\ell \in \{1, \dotsc, k-1\}\) be such that \(\ell > k - n\). If for every \(u \in C^\infty_c(\R^n; V)\)
\[
  \norm{D^{\ell}u}_{L^\frac{n}{n-(k-\ell)}} \le C \norm{A(D) u}_{L^1},
\]
then 
\(A(D)\) is canceling.
\end{proposition}

In this statement the operator is assumed to be elliptic, which is not necessary for the estimate when \(\ell< k-1\). We do not have any examples that show that this assumption is necessary:
\begin{openproblem}
Does proposition~\ref{propositionCancelingNecessary} remain true without the ellipticity assumption?
\end{openproblem}

\begin{remark}
\label{remarkCancelingLinfty}
Proposition~\ref{propositionCancelingNecessary} does not cover the case \(\ell=n-k\). 
In the case \(n = 1\), for every \(k \in \N_*\) the homogeneous linear differential operator \(A(D)\) defined for \(\xi \in \R\) by \(A(\xi)=\xi^k\) is elliptic but not canceling. Nonetheless, for every \(u \in C^\infty_c(\R)\), 
\[
 \norm{u^{(k-1)}}_{L^\infty} \le \norm{u^{(k)}}_{L^1}.
\]
We did not find higher-dimensional examples.
\end{remark}

\begin{proof}[Proof of proposition~\ref{propositionCancelingNecessary}]
Let \(e \in \bigcap_{\xi \in \R^n \setminus \{0\}} A(\xi)[V]\). Since for every \(\xi \in \R^n \setminus \{0\}\),  \(A(\xi)\) is one-to-one,
the function \(U : \R^n \setminus \{0\} \to V\) defined for each \(\xi \in \R^n \setminus \{0\}\)
\[
  A(\xi)\bigl[U(\xi)\bigr]=e
\]
is smooth. This can be seen by the implicit function theorem or by the formula \(U(\xi)=\bigl(A(\xi)^*\circ A(\xi)\bigr)^{-1} \circ A(\xi)^*[e]\). Since \(A(\xi)\) is homogeneous of degree \(k\), for every \(\xi \in \R^n \setminus \{0\}\) and \(t \in \R \setminus\{0\}\), 
\[
  U(t\xi)=t^{-k} U(\xi).
\]
Choose now a function \(\psi \in C^\infty( \R^n) \) such that \(\supp \widehat{\psi} \subset B_2(0)\) and \(\widehat{\psi}=1\) on \(B_{1/2}(0)\). For \(\lambda > 0\), define \(\psi_\lambda : \R^n \to \R\) for \(x \in \R^n\) by \(\psi_\lambda(x)=\lambda^n \psi(\lambda x)\), and define \(u_\lambda : \R^n \to V\) such that for each \(\xi \in \R^n\), 
\[
  \widehat{u_\lambda}(\xi)=(2\pi i)^{-k}\bigl(\widehat{\psi_\lambda}(\xi)-\widehat{\psi_{1/\lambda}}(\xi)\bigr)U(\xi).
\]
If \(\lambda > 2\), \(\supp (\widehat{\psi_\lambda}-\widehat{\psi_{1/\lambda}}) \subset B_{2\lambda}(0) \setminus B_{1/(2\lambda)}(0)\). Hence, \(u_\lambda\) is well-defined and belongs to the Schwartz class of fast decaying smooth functions.

We now claim that for every \(\lambda > 2\),
\begin{equation}
\label{inequlambda}
  \norm{D^{\ell} u_\lambda }_{L^\frac{n}{n - (k - \ell)}} \le C \norm{A(D) u_\lambda}_{L^1}.
\end{equation}
To see this, consider a function \(\varphi \in C^\infty_c(\R^n)\) such that \(\varphi=1\) on \(B_1(0)\). For \(R > 0\), define \(\varphi_R : \R^n \to \R\) for \(x \in \R^n\) by \(\varphi_R(x)=\varphi(x/R)\). By hypothesis, for every \(R > 0\), 
\[
  \norm{D^{\ell} (\varphi_R u_\lambda) }_{L^\frac{n}{n - (k - \ell)}} \le C \norm{A(D) (\varphi_R u_\lambda)}_{L^1}.
\]
By letting \(R \to \infty\), we obtain \eqref{inequlambda}.

Now, by definition of \(u_\lambda\) and the choice of \(e\), one has
\begin{equation}
\label{eqADulambda}
  A(D)u_\lambda=(\psi_\lambda-\psi_{1/\lambda})e,
\end{equation}
and therefore, 
\begin{equation}
\label{inequlambdaL1}
  \norm{A(D)u_\lambda}_{L^1} \le 2 \norm{\psi}_{L^1}. 
\end{equation}
On the other hand, for every \(\alpha \in \N^n\), \(\lambda >2\) and \(x \in \R^n\)
\[
  \partial^\alpha u_\lambda(x)
  =\int_{\R^n} e^{2\pi i\xi \cdot x}  \bigl(\widehat{\psi}(\xi/\lambda)-\widehat{\psi}(\lambda \xi)\bigr)
\,(2\pi i)^{\abs{\alpha}-k}\xi^\alpha U(\xi)\,d\xi.
\]
By writing for every \(\xi \in \R^n\)
\[
 \widehat{\psi}(\xi/\lambda)-\widehat{\psi}(\lambda \xi)=-\int_{1/\lambda}^\lambda \frac{\xi}{t} \cdot \nabla \widehat \psi \Bigl(\frac{\xi}{t}\Bigr)\,\frac{dt}{t},
\]
we have, by Fubini's theorem, 
\[
  \partial^\alpha u_\lambda(x)=\int_{1/\lambda}^\lambda w^\alpha(tx) t^{n-(k-\abs{\alpha})} \frac{dt}{t}.
\]
where \(w^\alpha : \R^n \to V\) is defined for \(x \in \R^n\) by 
\[
w^\alpha(x)=-(2\pi i)^{\abs{\alpha}-k}\int_{\R^n} e^{2\pi i \xi \cdot x}  \xi \cdot \nabla \widehat \psi (\xi)\,\xi^\alpha U(\xi)\, d\xi,  
\]
Since \(w_\alpha\) decays fast at infinity, if \(\abs{\alpha} > k - n \) and \(x \in \R^n \setminus \{0\}\), the limit
\begin{equation}
\label{defualpha}
  u^\alpha(x)=\lim_{\lambda \to \infty} \partial^\alpha u_\lambda(x)=\int_{0}^\infty w^\alpha(tx) t^{n-(k-\abs{\alpha})} \frac{dt}{t}
\end{equation}
is well-defined.

Assume by contradiction that there exists \(\alpha \in \N^n\) such that \(\abs{\alpha}=\ell\) and \(u_\alpha \not \equiv 0\). For every \(x \in \R^n\) and \(t > 0\), one has by \eqref{defualpha}
\begin{equation}
\label{equalphahomog}
 u^\alpha(t x)=\frac{u^\alpha(x)}{t^{n-(k-\ell)}}.
\end{equation}
Since \(u^\alpha \not \equiv 0\), this implies that
\[
\label{equalphainfty}
\int_{\R^n} \abs{u^\alpha}^\frac{n}{n-(k-\ell)}=\infty.
\]
By Fatou's lemma we have
\[
  \liminf_{\lambda \to \infty} \int_{\R^n} \abs{\partial^\alpha u_\lambda}^\frac{n}{n-(k-\ell)} \ge \int_{\R^n} \abs{u^\alpha}^\frac{n}{n-(k-\ell)}=\infty,
\]
in contradiction with \eqref{inequlambda} and \eqref{inequlambdaL1}.

We have thus \(u^\alpha\equiv 0\) for every \(\alpha \in \N^n\) with \(\abs{\alpha}=\ell\).
For each \( x \in \R^n \setminus \{0\}\), \(\lambda > 0\) and \(\alpha \in \N^n\) with \(\abs{\alpha}=\ell\), we have by \eqref{defualpha}
\[
\begin{split}
 \abs{\partial^\alpha u_\lambda(x)} &\le \int_0^\infty \abs{w^\alpha(tx)} t^{n-(k-\ell)}\frac{dt}{t}\\
&=\frac{1}{\abs{x}^{n-(k-\ell)}} \int_0^\infty \Bigl\lvert w^\alpha\Bigl(t\frac{x}{\abs{x}}\Bigr)\Bigr\rvert t^{n-(k-\ell)}\frac{dt}{t},
\end{split}
\]
and therefore 
\[
 \abs{\partial^\alpha u_\lambda(x)} \le \frac{C}{\abs{x}^{n-(k-\ell)}}.
\]
By Lebesgue's dominated convergence theorem, \(D^\ell u_\lambda \to 0 \) in \(L^1_{\mathrm{loc}}(\R^n)\).
Taking now \(\zeta \in C^\infty_c(\R^n)\), we obtain by a suitable integration by parts that 
\begin{equation}
\label{limzetae}
 \lim_{\lambda \to \infty} \int_{\R^n} \zeta A(D)u_\lambda=0.
\end{equation}
On the other hand, in view of \eqref{eqADulambda}, one has
\[
  \int_{\R^n} \zeta A(D)u_\lambda=\int_{\R^n} (\psi_\lambda-\psi_{1/\lambda})\zeta e,
\]
whence
\[  
  \lim_{\lambda \to \infty} \int_{\R^n} \zeta A(D)u_\lambda=\zeta(0) e. 
\]
Since this should hold for every \(\zeta \in C^\infty_c(\R^n)\), this implies in view of \eqref{limzetae} that \(e=0\).
\end{proof}

\section{Characterization and examples of canceling operators}

\label{sectionExamples}

\subsection{Analytic characterization of elliptic canceling operators}
We have seen in proposition~\ref{propositionEquivalentCocanceling} that the cocanceling condition is equivalent with a property of the vector fields that are in its kernel. For elliptic canceling operators, the same methods allow to characterize canceling operators by properties of the image of vector fields.

\begin{proposition}
\label{propositionEquivalentCanceling}
Let \(A(D)\) be a homogeneous differential operator of order \(k\) on \(\R^n\) from \(V\) to \(E\).
If \(A(D)\) is elliptic, the following are equivalent
\begin{enumerate}[(i)]
 \item\label{itCanceling} \(A(D)\) is canceling, 
 \item\label{itCancelingVanishL1} for every \(u \in L^1_{\mathrm{loc}}(\R^n; V)\), if \(A(D)u \in L^1(\R^n; E)\), then 
\[\int_{\R^n} A(D)u = 0,
\]
\item\label{itCancelingVanishCinfty} for every \(u \in C^\infty(\R^n; V)\), if 
\(\supp A(D)u\) is compact, then 
\[
  \int_{\R^n} A(D) u = 0,
\]
\end{enumerate}
\end{proposition}

If for every \(j \in \{0, \dotsc, k - 1\}\), 
\[\lim_{\abs{x} \to \infty} \abs{D^j u (x)} \abs{x}^{n - j} = 0 \] then 
\[
  \int_{\R^n} A(D) u = 0,
\]
for any operator \(A(D)\). It is thus crucial that no decay assumption is imposed on \(u\) in \eqref{itCancelingVanishCinfty}.

\begin{proof}
First note that since \(A(D)\) is elliptic, proposition~\ref{propositionConstructionL} applies and yields a homogeneous differential operator \(L(D)\) on \(\R^n\) from \(E\) to \(F\). This operator \(L(D)\) is cocanceling if and only if \(A(D)\) is canceling.

Let us now prove that \eqref{itCanceling} implies \eqref{itCancelingVanishL1}. Let \(u \in L^1_{\mathrm{loc}}(\R^n; E)\) be such that \(A(D)u \in L^1(\R^n; E)\). By construction of \(L(D)\) 
\[
  L(D) \bigl(A(D)u\bigr)=0.
\]
Since by assumption \(L(D)\) is cocanceling, in view of proposition~\ref{propositionEquivalentCocanceling} \eqref{itVanishL1},
\[
  \int_{\R^n} A(D)u=0.
\]

It is clear that \eqref{itCancelingVanishL1} implies \eqref{itCancelingVanishCinfty}. Assume now that \eqref{itCancelingVanishCinfty} holds. Let \(f \in C^\infty_c(\R^n; E)\) be such that \(L(D)f=0\). 

This latter condition allows to define \(w : \R^n \to \Lin^k(\R^n; V)\) such that its Fourier transform \(\widehat{w}\) satisfies for every \(\xi \in \R^n\)
\[
  A(\xi)\bigl[\widehat{w}(\xi)[v_1, \dotsc, v_k] \bigr]=(\xi \cdot v_1) \dotsm (\xi \cdot v_k) \widehat{f}(\xi).
\]
Since \(A(D)\) is elliptic and \(f\) is smooth, \(w\) is smooth. 
Write now
\[
 u (x) = \int_0^1 w (t x) [x, \dotsc, x] \frac{(1 - t)^{k - 1}}{(k - 1)!} \, d t,
\]
so that \(D^k u = w\) and hence \(A(D)u=f\).
By assumption we have that 
\[
  \int_{\R^n}f=\int_{\R^n} A(D)u=0.
\]
In view of  proposition~\ref{propositionEquivalentCocanceling}, we have proved that \(L(D)\) is cocanceling. This allows to conclude that \(A(D)\) is canceling.
\end{proof}

\subsection{Equivalence between cancellation and the Bourgain--Brezis algebraic condition}

J.\thinspace Bourgain and H.\thinspace Brezis \cite{BB2007}*{theorem 25}
have proved the estimate 
\[
  \norm{u}_{L^\frac{n}{n-1}} \le C \norm{A (D) u}_{L^1}
\]
for an elliptic operator \(A(D)\) under the structural condition that there exist a basis \(e_1, \dotsc, e_\ell\) of \(E\) and vectors \(\xi_1, \dotsc, \xi_\ell \in \R^n \setminus\{0\}\) such that for every \(i \in \{1, \dotsc, \ell\}\), 
\(
  e_i \perp A(\xi_i)[V]
\)
\footnote{The statement and the proof \cite{BB2007}*{theorem 25} are written for \(\dim V=n\);  the arguments adapt straightforwardly when \(\dim V \ne n\).}. This condition is in fact equivalent with the cancellation condition

\begin{proposition}
\label{propBBCanceling}
Let \(A(D)\) be a homogeneous differential operator on \(\R^n\) from \(V\) to \(E\).
The operator \(A(D)\) is canceling if and only if 
\[
  \operatorname{span} \bigcup_{\xi \in \R^n \setminus \{0\}} \bigl(A(\xi)[V]^\perp\bigr)=E.
\]
\end{proposition}

\begin{proof}
For every \(\xi \in \R^n\), since \(A(\xi)\) is a linear operator, one has \(e \in A(\xi)[V]\) if and only if for every \(f \in A(\xi)[V]^\perp\), \(f \cdot e=0\). Therefore, \(e \in \bigcap_{\xi \in \R^n \setminus\{0\}} A(\xi)[V]\)
if and only if for every \(\xi \in \R^n \setminus \{0\}\) and for every \(f \in A(\xi)[V]^\perp \), \( f \cdot e = 0\). We have thus
\[
  \bigcap_{\xi \in \R^n\setminus\{0\}} A(\xi)[V]=\Bigl(\bigcup_{\xi \in \R^n \setminus \{0\} } \bigl(A(\xi)[V]^\perp\bigr)\Bigr)^\perp.
\]
Hence, one has that 
\[
  \bigcap_{\xi \in \R^n\setminus\{0\}} A(\xi)[V]=\{0\}
\]
if and only if 
\[
\operatorname{span} \biggl(\bigcup_{\xi \in \R^n \setminus \{0\}} \bigl(A(\xi)[V]^\perp\bigr)\biggr)=E,
\]
which is the statement that we wanted to prove.
\end{proof}

\begin{remark}
The same argument shows that a linear homogeneous differential operator \(L(D)\) on \(\R^n\) from \(V\) to \(E\) is cocanceling if and only if
\[
  \operatorname{span} \biggl(\bigcup_{\xi \in \R^n \setminus\{0\}} \bigl(\ker L(\xi)^\perp\bigr)\biggr)=E, 
\]
or equivalently
\[
  \operatorname{span} \biggl(\bigcup_{\xi \in \R^n \setminus\{0\}}  L(\xi)^*[V]\biggr)=E.
\]
\end{remark}

\subsection{First-order canceling operators}
We shall now give explicit examples of canceling operators.

\subsubsection{Gradient operator}
The simplest example is the gradient operator:

\begin{proposition}
\label{propositionGradientCanceling}
Let \(A(D)\) be the homogeneous linear differential operator of order \(1\) on \(\R^n\) from \(\R\) to \(\R^n\) defined for \(\xi \in \R^n\) by 
\[
  A(\xi)=\xi.
\] 
The operator \(A(D)\) is elliptic.\\
The operator \(A(D)\) is canceling if and only if \(n \ge 2\).
\end{proposition}
\begin{proof}
For every \(\xi \in \R^n\) \(A(\xi)[\R]=\R \xi\), therefore \(\bigcap_{\xi \in \R^n \setminus \{0\}} A(\xi)[\R]=\{0\}\) if \(n \ge 2\) and \(\bigcap_{\xi \in \R^n \setminus \{0\}} A(\xi)[\R]=\R\) if \(n=1\).
\end{proof}

\subsubsection{Symmetric derivative}
The symmetric derivative operator appearing in the Korn--Sobolev inequality \eqref{KornSobolevL1} is also an elliptic canceling operator.
Recall that \(S^2\R^n\) is the space of symmetric bilinear forms on \(\R^n\).

\begin{proposition}
\label{propositionCancelingSymmetricDerivative}
Let \(A(D)\) be the homogeneous linear differential operator of order \(1\) on \(\R^n\) from \(\R^n\) to \(S^2\R^n\) defined for \(\xi \in \R^n\), \(v \in \R^n\) and \(w, z \in \R^n\) by 
\[
  A(\xi)[v](w, z)=\frac{1}{2}\bigl((\xi \cdot w)(v \cdot z)+(\xi \cdot z)(v \cdot w) \bigr).
\] 
The operator \(A(D)\) is elliptic.\\
The operator \(A(D)\) is canceling if and only if \(n \ge 2\).
\end{proposition}
\begin{proof}
The operator \(A(D)\) is elliptic: assume that \(v \in \R^n\) and \(\xi \in \R^n \setminus \{0\}\) are such that for every \(w, z \in \R^n\), \(A(\xi)[v](w, z)=0\). In particular, for every \( w \in \R^n \), 
\[
 A(\xi)[v](w, w)=(\xi \cdot w)(v \cdot w).
\]
We have thus for every \(w \in \R^n\) such that \( \xi \cdot w \ne 0\), \(v \cdot w=0\). Since such \(w \) span \(\R^n\), we have proved that \(A(D)\) is elliptic.

Now we prove that \(A(D)\) is canceling when \(n \ge 2\). Let \(e \in  \bigcap_{\xi \in \R^n \setminus \{0\}} A(\xi)[\R^n]\). For every \(w \in \R^n\), choosing \(\xi \in \R^n \setminus \{0\}\) such that \(\xi \cdot w = 0\), one has for every \(v \in \R^n\), 
\[
  A(\xi)[v](w, w)=0
\]
and therefore, \(e(w, w)=0\). Since \(w \in \R^n\) is arbitrary and \(e\) is symmetric, we conclude that \(e=0\).
\end{proof}

The application of theorem~\ref{theoremOrderk} yields the Korn--Sobolev inequality \eqref{KornSobolevL1}. The application of theorem~\ref{theoremOrderkFractional} would yield fractional Korn--Sobolev inequalities.

\medbreak

This example has a counterpart for the symmetric (or inner) derivative of a symmetric multilinear forms
\cite{Sharafutdinov}*{p.\thinspace 25}

\begin{proposition}
Let \(A(D)\) be the homogeneous linear differential operator of order \(1\) on \(\R^n\) from \(S^k\R^n\) to \(S^{k+1}\R^n\) defined for \(v\in S^k \R^n\), \(\xi \in \R^n\) and \(w_1, \dotsc, w_{k+1} \in \R^n\) by 
\begin{multline*}
  A(\xi)[v](w_1, \dotsc, w_{k+1})\\=\tfrac{1}{k+1} \bigl((\xi \cdot w_1)v(w_2, \dotsc, w_{k+1})
+(\xi \cdot w_2)v(w_1, w_3, \dotsc, w_{k+1})\\
+\dotsb+(\xi \cdot w_{k+1})v(w_1, \dotsc, w_{k})\bigr)
\end{multline*}
The operator \(A(D)\) is elliptic.\\
The operator \(A(D)\) is canceling if and only if \(n \ge 2\).
\end{proposition}

\begin{proof}
For the ellipticity,  assume that \(v \in S^k\R^n\) and \(\xi \in \R^n \setminus \{0\}\) are such that for every \(w_1, \dotsc, w_{k+1} \in \R^n\), \(L(\xi)[v](w_1, \dotsc, w_{k+1})=0\). In particular, for every \(\xi \in \R^n\), 
\[
 A(\xi)[v](w, \dotsc, w)=(\xi \cdot w) v(w, \dotsc, w)=0.
\]
Therefore, for every \(w \in \R^n\) such that \(\xi \cdot w \ne 0\), \(v(w, \dotsc, w)=0\). This implies that \(v=0\).

Now we prove that \(A(D)\) is canceling when \(n \ge 2\). Let \(e \in  \bigcap_{\xi \in \R^n \setminus \{0\}} A(\xi)[S^k\R^n]\). For every \(w \in \R^n\), choosing \(\xi \in \R^n \setminus \{0\}\) such that \(\xi \cdot w = 0\), one has for every \(v \in S^k\R^n\), 
\[
  A(\xi)[v](w, w, \dotsc, w) = 0
\]
and therefore, \(e(w, \dotsc, w)=0\). Since \(w \in \R^n\) is arbitrary and \(e\) is symmetric, we conclude that \(e=0\). 
\end{proof}

\subsubsection{Exterior derivative}
We now turn to the study of canceling operators appearing in the framework of exterior differential calculus.

\begin{proposition}
\label{propositionCancelingHodge}
Let \(\ell \in \{1, \dotsc, n-1\}\) and let \(A(D)=(d, d^*)\) be the homogeneous linear differential operator of order \(1\) on \(\R^n\) from \(\bigwedge^\ell \R^n\) to \(\bigwedge^{\ell+1}\R^n \times \bigwedge^{\ell-1}\R^n\) such that for every \(\xi \in \R^n\) and \(v \in \bigwedge^\ell \R^n\)
\[
  A(\xi)[v]=\bigl(\xi \wedge v, * (\xi \wedge * v)\bigr).
\]
The operator \(A(D)\) is elliptic.\\
The operator \(A(D)\) is canceling if and only if \(\ell \in \{2, \dotsc, n-2\}\).
\end{proposition}
\begin{proof}
The ellipticity follows from the Lagrange identity \( \abs{v}^2 \abs{\xi}^2 = \abs{\xi \wedge v}^2+\abs{*(\xi \wedge * v)}^2\).

For the cancellation, if \((f, g) \in \bigcap_{\xi \in \R^n \setminus \{0\}} A(\xi)[\bigwedge^\ell \R^n]\), one should have for every \(\xi \in \R^n\), \(\xi \wedge f  =0\) and \(\xi \wedge * g=0\). Since \(2 \le \ell \le n-2\), this implies that \(f=0\) and \(g=0\). 
\end{proof}

As a consequence of proposition~\ref{propositionCancelingHodge}, one gets the Hodge--Sobolev inequality \eqref{HodgeSobolevL1}.

\subsubsection{Directional derivatives of vector fields}

One has also a general construction to control a vector field by directional derivatives of some components

\begin{proposition}
\label{propositionDerivativesSpan}
Let \(m=\dim V\).
Consider a family of \(n+m-1\) \(n\)--wise linearly independent vectors \((\eta_i)_{1 \le i \le n+m-1}\) of\/ \(\R^n\) and \(m\)--wise linearly independent vectors \((w_i)_{1 \le i \le n+m-1}\) of \(V\) and define for \(\xi \in \R^n\) and \(v \in V\), 
\[
  A(\xi)[v]=\bigl( (\eta_1 \cdot \xi)(w_1 \cdot v), \dotsc, (\eta_{m+n-1} \cdot \xi)(w_{m+n-1} \cdot v)\bigr).
\]
The operator \(A(D)\) is elliptic.\\
The operator \(A(D)\) is canceling if and only if \(n \ge 2\).
\end{proposition}

This construction is due to D.\thinspace G.\thinspace de Figueiredo \cite{F1963}*{inequality (K)} in the framework of \(L^2\) estimates. It was introduced by the author in the context of generalized Korn--Sobolev inequalities \cite{BB2007}*{remark 16}.

\begin{proof}[Proof of proposition~\ref{propositionDerivativesSpan}]
Let us first show that \(v\) is elliptic. 
Let \(\xi \in \R^n \setminus \{0\}\) and \(v\in V\) be such that \(A(\xi)[v]=0\).
Since the vectors \((\eta_i)_{1 \le i \le n+m-1}\) are \(n\)--wise linearly independent, there is an increasing sequence of indices \(i_1, \dotsc, i_m\) such that for every \(j \in \{1, \dotsc, m\}\),
\(
 (\eta_{i_j} \cdot \xi)\ne 0.
\)
Therefore, for every \(j \in \{1, \dotsc, m\}\), 
\(
 (w_{i_j} \cdot v) = 0.
\)
Since the vectors \(w_{j_1}, \dotsc, w_{j_m}\) form a basis of \(V\), we conclude that \(v=0\).

For the cancellation, assume that \(e \in \bigcap_{\xi \in \R^n \setminus \{0\}} A(\xi)[V]\). By taking \(\xi_i \in \R^n \setminus\{0\}\) such that \(\xi_i \cdot \eta_i=0\), we have that for every \(e \in A(\xi_i)[V]\), \(e_i=0\). Since \(e \in \bigcap_{i=1}^{n+m-1} A(\xi_i)[V]\), we conclude that \(e=0\). We have thus proved that \(A(D)\) is canceling.
\end{proof}

By theorem~\ref{theoremOrderk}, this yields

\begin{proposition}
Let \(m=\dim V\). Consider a family of \(n+m-1\) \(n\)--wise linearly independent vectors \((\eta_i)_{1 \le i \le n+m-1}\) of\/ \(\R^n\) and \(m\)--wise linearly independent vectors \((w_i)_{1 \le i \le n+m-1}\) of \(V\). For every \(u \in C^\infty_c(\R^n; V)\)
\[
  \norm{u}_{L^{n/(n-1)}} \le C \sum_{i=1}^{m+n-1}\norm{w_i \cdot Du[\eta_i]}_{L^1}. 
\]
\end{proposition}

\subsubsection{Minimizing the number of components of the derivative}

The previous example shows that a vector field \(u \in C^\infty_c(\R^n; \R^m)\) can be estimated by \(n+m-1\) directional derivatives of components. One may wonder whether it is possible to use less derivatives \cite{BB2007}*{open problem 3}.

For a lower bound we have
\begin{proposition}
\label{propositionBoundsDimE}
Assume that \(A(D)\) is a differential operator of order \(1\) on \(\R^n\) from \(V\) to \(E\) that is canceling and elliptic. Then 
\(\dim E > \dim V\) and \(\dim E \ge n\).
\end{proposition}
\begin{proof}
Since \(A(D)\) is canceling, there exists \(\xi \in \R^n\) such that \(A(\xi)[V] \ne E\). Since \(A(D)\) is elliptic, this implies that \(\dim E > \dim V\).

Next fix \(v \in V\) and consider the linear map \(T : \R^n \to E\) defined by \(T(\xi)=A[\xi](v)\). Since \(A(D)\) is elliptic, \(\ker T=\{0\}\). Therefore, \(n=\dim T(\R^n) \le \dim E\).
\end{proof}

If we define \(l^*(n, m)\) to be the minimal dimension \(l\) such that there is a canceling elliptic linear differential operator on \(\R^n\) from \(\R^m\) to \( \R^l\), we have by propositions~\ref{propositionDerivativesSpan} and \ref{propositionBoundsDimE}
\begin{equation}
\label{boundtriviallstar}
 \max(n, m+1)  \le l^*(n, m) \le m+n-1.
\end{equation}
In particular, the construction of proposition~\ref{propositionDerivativesSpan} is optimal if \(m=1\) (the scalar case) or \(n=2\).

The Hodge--Sobolev estimate for \(n=4\) and \(\ell=2\) uses less components: one has \(V=\bigwedge^2 \R^4\), and thus \(m=\dim V = 6\) whereas \(E=\bigwedge^1 \R^4 \times \bigwedge^3 \R^4\), so that \(\dim E = 8 < 9 = n+m-1\). We have thus \(7 \le l^*(4, 6) \le 8\).
In all the other cases the Hodge--Sobolev inequality does not allow to estimate with less components than \(n + \dim V - 1\). Indeed, one has
\(
  \dim \Bigl( \bigwedge^{\ell-1} \R^n \times \bigwedge^{\ell+1} \R^n\Bigr)=\dim \bigwedge^\ell \R^n + \binom{n-1}{\ell-2}+\binom{n-1}{\ell+1}.
\)
The condition to have the Hodge--Sobolev inequality is \(2 \le \ell \le n-2\). If we want to use less components than \(n+m-1\), we need to have
\(
  \binom{n-1}{\ell-2}+\binom{n-1}{\ell+1} < n-1.
\)
This is only possible if \(n=4\) and \(\ell=2\).

The Korn--Sobolev uses \(\dim E= \frac{n(n+1)}{2}\) components, which is always larger or equal to \(2n-1\).

There are now specific constructions that work in some cases. Let \(\quat \simeq \R^4\) be the algebra of quaternions

\begin{proposition}
Let \(A(D)\) the homogeneous linear differential operator of order \(1\) on \(\R^4\) from 
\(V=\{x \in \quat \st \realpart x = 0\}\) to \(\quat\), defined for every \(v \in V\) and \(\xi \in \R^4\) by
\[
  A(\xi)[v]=\xi v.
\]
The operator \(A(D)\) is canceling and elliptic.
\end{proposition}
Alternatively, writing \(\xi=(\xi_1, \xi'')\in \R^1 \times \R^3\), one has \(A(\xi)[v]=(-\xi'' \cdot v, \xi_1 v+\xi'' \times v)\).

\begin{proof}
Since the multiplication of quaternions is invertible, \(A(D)\) is elliptic. 

For the cancellation property, for every \(v \in V\) and \(\xi \in \R^4 \setminus \{0\}\), one has
\(
  \realpart \bigl(\xi^{-1} A(\xi)[v]\bigr)=\realpart v =0.
\)
Hence, if \(e \in A(\xi)[V]\) for every \(\xi \in \R^4 \setminus \{0\}\), one has for every \(\xi \in \R^4 \setminus \{0\}\), 
\(
  \realpart \xi^{-1} e = 0,
\)
whence \(e=0\).
\end{proof}

This gives the estimate for every \(u \in C^\infty_c(\R^4; \R^3)\), 
\[
  \norm{u}_{L^{4/3}} \le C\bigl(\norm{\Div'' u}_{L^1}+\norm{\partial_1u + \Curl'' u}_{L^1}\bigr),
\]
where \(\Div'' u\) and \(\Curl'' u\) denote respectively the divergence and the curl with respect to the last three variables.

The previous example shows that \(l^*(4, 3)=4\). The same construction can be made with the octonions and allows to control a vector field from \(\R^8\) to \(\R^7\), showing that \(l^*(8, 7) = 8\). If the same construction is made with complex numbers instead of the octonions, one recovers the limiting Sobolev inequality for scalar functions on \(\R^2\).

The previous construction also allows to show again that \(l^*(4, 6) \le 8\) and to show that that \(l^*(8, 7j) \le 8j\); which is an improvement of the previous bound \eqref{boundtriviallstar} when \(j \le 6\).

\subsection{Second-order estimates}
We now give example of second-order canceling elliptic operators and of application of theorem~\ref{theoremOrderk}.

\subsubsection{Splitting the Laplace--Beltrami operator}
The Laplacian is never a canceling operator. However, when split into two parts, it might become canceling

\begin{proposition}
Let \(n \ge 2\), \(\ell \in \{1, \dotsc, n-1\}\) and let \(A(D)\) be the homogeneous linear differential operator of order \(2\) from \(\bigwedge^\ell \R^n\) to \(\bigwedge^\ell \R^n \times \bigwedge^\ell \R^n\) defined for \(u \in C^\infty(\R^n; \bigwedge^\ell \R^n)\) by 
\[
  A(D)[u]=(dd^*u, d^*du).
\]
The operator \(A(D)\) is elliptic and canceling.
\end{proposition}
\begin{proof}
Since \(dd^*+d^*d=\Delta\) is elliptic, \(A(D)\) is clearly elliptic. 
For the cancellation, let \(f, g \in \bigcap_{\xi \in \R^n \setminus\{0\}} A(\xi)[V]\). One has for every \(\xi \in \R^n\), \(\xi \wedge f = 0\) and \(\xi \wedge * g = 0\). Since \(f, g \in \bigwedge^{\ell} \R^n\) with \(\ell \in \{1, \dotsc, n-1\}\), this implies that \(f=g=0\).
\end{proof}

\begin{corollary}
Let \(n \ge 2\), \(\ell \in \{1, \dotsc, n-1\}\). For every \(u \in C^\infty_c(\R^n; \bigwedge^\ell \R^n)\), 
\[
  \norm{Du}_{L^{\frac{n}{n-1}}} \le C \bigl(\norm{dd^*u}_{L^1}+\norm{d^*du}_{L^1}\bigr).
\]
\end{corollary}

\subsubsection{Linearly independent collections of operators}
A similar situation can be observed for a collection of scalar operators

\begin{proposition}
\label{propExample2dOrder}
Let \((w_i)_{1 \le i \le m+1}\) be \(m\)--wise linearly independent vectors of \(V\) and \((a_i)_{1 \le i \le m+1}\) be quadratic forms on \(\R^n\) such that 
if for every \(i, j \in \{1, \dotsc, m+1\}\) with \(i < j\), then
\[
  \bigl\{\xi \in \R^n  \st a_i(\xi)=0\bigr\} \cap \bigl\{\xi \in \R^n  \st a_j(\xi)=0\bigr\}=\{0\}
\]
and for every \(i \in \{1, \dotsc, m+1\}\)
\[
 \bigl\{\xi \in \R^n  \st a_i(\xi)=0\bigr\} \ne \{0\}.
\]
Define 
\[
 A(\xi)[v]=\bigl(a_1(\xi)(w_1 \cdot v), \dotsc, a_{m+1}(\xi)(w_{m+1} \cdot v)\bigr).
\]
The operator \(A(D)\) is elliptic and canceling.
\end{proposition}

\begin{proof}
We first prove that \(A(D)\) is elliptic. Indeed, if \(A(\xi)[v]=0\), then there exists \(j \in \{1, \dotsc, m+1\}\) such that for \(i \ne j\), 
\(a_i(\xi)\ne 0\). We have thus for \(i \ne j\), \(w_i \cdot v=0\), which implies \(v=0\).

Now we show that \(A(D)\) is canceling. For every \(i \in \{1, \dotsc, m+1\}\), one can find \(\xi \in \R^n \setminus \{0\}\) such that 
\( a_i(\xi)=0\).
This proves that if \(e \in \bigcap_{\xi \in \R^n \setminus\{0\}} A(\xi)[V]\), then \(e_i=0\). Since this is true for every \(i \in \{1, \dotsc, m+1\}\), \(A(D)\) is canceling.
\end{proof}

The  construction of proposition~\ref{propExample2dOrder} is always possible given any \(n \ge 2\) and \(V\). Indeed take \(\xi_1, \dotsc, \xi_{m+1}\) to be unit vectors of \(\R^n\) such that \(\abs{\xi_i \cdot \xi_j} = 0\) if \(i \ne j\) and set for \(\xi \in \R^n\), \(a_i(\xi)=\abs{\xi}^2-(\xi_i \cdot \xi)^2\). 
Since for an elliptic canceling linear differential operator \(A(D)\) on \(\R^n\) from \(V\) to \(E\) one needs to have \(\dim E > \dim V\), this construction is the most economic in terms of the number of components of the second order derivative that are taken.

In view of theorem~\ref{theoremOrderk}, for every \(u \in C^\infty_c(\R^n; V)\),
\[
 \norm{Du}_{L^\frac{n}{n-1}} \le C \Bigl(\sum_{i=1}^{m+1} \norm{a_i(D) w_i \cdot u)}_{L^1}\Bigr).
\]
In particular for every \(u \in C^\infty_c(\R^2)\), 
\begin{equation}
\label{ineqL2Ornstein}
 \norm{\nabla u}_{L^2} \le C \bigl(\norm{\partial_1^2u}_{L^1}+\norm{\partial_2^2 u}_{L^1}\bigr).
\end{equation}
This inequality is originally due to V.\thinspace A.\thinspace Solonnikov \cite{Solonnikov1975}*{theorem 3}.
This estimate is quite striking because there is no estimate of the form
\begin{equation}
\label{ineqL2DeltaL1}
 \norm{\nabla u}_{L^2} \le C\norm{\partial_1^2u+\partial_2^2 u}_{L^1}
\end{equation}
as one can see by inspection of the fundamental solution of \(-\Delta\) on \(\R^2\) nor of the form
\begin{equation}
\label{ineqL1Ornstein}
 \norm{D^2 u}_{L^1} \le C\bigl(\norm{\partial_1^2u}_{L^1}+\norm{\partial_2^2 u}_{L^1}\bigr).
\end{equation}
(this was the original motivation of D. Ornstein's work \cite{Ornstein1962}). The inequality \eqref{ineqL2Ornstein} also explains why the construction of D.\thinspace Ornstein to disprove \eqref{ineqL1Ornstein} had to go beyond the study of the fundamental solutions, as one does to disprove \eqref{ineqL2DeltaL1}.

\section{Partially canceling operators}
\label{sectionPartCanceling}

\subsection{Partially canceling operators}
If an operator \(A(D)\) is not canceling, there is still a weaker inequality.

\begin{theorem}
\label{theoremPartiallyCancelingKernel}
Let \(n \ge 2\), let \(A(D)\) be an elliptic linear homogeneous differential operator on \(\R^n\) from \(V\) to \(E\) and let \(T \in \Lin(E; F)\).
The estimate 
\[
  \norm{D^{k-1} u}_{L^{\frac{n}{n-1}}} \le C\norm{A(D)u}_{L^1}
\]
holds for every \(u \in C^\infty_c(\R^n; V)\) such that \(T \circ A(D) u = 0\) if and only if 
\[
 \bigcap_{\xi \in \R^n \setminus\{0\}} A(\xi)[V] \cap \ker T = \{0\}.
\]
\end{theorem}

\begin{remark}
The estimate does not imply ellipticity. Indeed, take \(A(D)\) on \(\R^2\) from \(\R^2\) to \(\R^3\) defined by \(A(D)[u]=(\partial_1 u_1, \partial_2 u_1, \partial_2 u_2)\) and \(T \in \Lin(\R^3; \R^2)\) defined by \(T(v)=(v_1, v_3)\). If  \(u \in C^\infty_c(\R^2; \R^2)\) and \(T \circ A(D)u=0\), then \(u=0\). Therefore the estimate holds trivially. On the other hand \(A(D)\) is not elliptic as \(A(1, 0)[(0, 1)]=0\).
\end{remark}

\subsection{Estimates for partially cocanceling operators}
\label{sectionPartCocanceling}
In order to prove theorem~\ref{theoremPartiallyCancelingKernel} we shall need an extension of theorem~\ref{theoremCocanceling} to partially cocanceling operators.

\begin{proposition}
\label{propositionPartialCocanceling}
Let \(L(D)\) be a homogeneous linear differential operator of order \(k\) on \(\R^n\) from \(E\) to \(F\) and let \(Q \in \Lin(E; E)\) be a projector.
If
\[
  \ker Q = \bigcap_{\xi \in \R^n \setminus \{0\}} \ker L(\xi),
\]
then for every \(f \in L^1(\R^n; E)\) such that \(L(D)f=0\) and \(\varphi \in C^\infty_c(\R^n; E)\), 
\[
  \int_{\R^n} (Q \circ f) \cdot \varphi \le C \norm{Q \circ f}_{L^1} \norm{ D \varphi}_{L^n}.
\]
\end{proposition}

\begin{proof}
Define \(\Tilde{L}(D)\) to be the linear homogeneous differential operator on \(\R^n\) from \(Q(E)\) to \(F\) defined by restriction of \(L(D)\). Since \(\bigcap_{\xi \in \R^n \setminus \{0\}} \ker L(\xi) \subseteq \ker Q \), \(\Tilde{L}(D)\) is cocanceling. Moreover, since \(\ker Q \subset \bigcap_{\xi \in \R^n \setminus \{0\}} \ker L(\xi)\) and \(Q\) is a projector, for every \(\xi \in \R^n\), \((\id-Q)(E)=\ker Q \subseteq \ker L(\xi)\). Hence, one has \(L(\xi) = L (\xi) \circ Q = \Tilde{L}(\xi) \circ Q\). Assume now that \(L(D)f=0\). One has then \(\Tilde{L}(D)(Q \circ f)=0\). Since \(\Tilde{L}\) is cocanceling, theorem~\ref{theoremCocanceling} applies to \(Q \circ f\) and gives the estimate.
\end{proof}

There is a converse statement to proposition~\ref{propositionPartialCocanceling} 
\begin{proposition}
Let \(L(D)\) be a homogeneous linear differential operator from \(E\) to \(F\) and let \(Q \in \Lin(E; F)\).
If for every \(f \in L^1(\R^n; E)\) such that \(L(D)f=0\), one has \(Q \circ f \in \dot{W}^{-1, \frac{n}{n-1}}(\R^n; E)\), then 
\[
  \bigcap_{\xi \in \R^n \setminus \{0\}} \ker L(\xi) \subseteq \ker 
Q.
\]
\end{proposition}

\begin{proof}
Let \(e \in \bigcap_{\xi \in \R^n \setminus\{0\}} \ker L(\xi)\). By assumption if \(f \in L^1(\R^n; \R)\), one has \(fQ(e) \in \dot{W}^{-1, \frac{n}{n-1}}\), and then necessarily \(\int_{\R^n} f Q (e) = 0\). By choosing \(f\) such that \(\int_{\R^n} f = 1\), we conclude that \(Q(e)=0\).
\end{proof}

\subsection{An example of partially cocanceling operator operator}
An example of partially cocanceling operator is given by the $\mathrm{Curl}\,\mathrm{Div}$ operator:

\begin{proposition}
\label{propositionBC}
Let \(L(D)\) be the homogeneous linear differential operator of order \(2\) on \(\R^n\) from \(\Lin(\R^n; \R^n)\) to \(\bigwedge^2 \R^n\) defined for \(\xi \in \R^n \simeq \bigwedge^1 \R^n\) and \(e \in \Lin(\R^n; \R^n) \) by
\[
  L(\xi)[e]=\xi \wedge e(\xi),
\]
One has 
\[
  \bigcap_{\xi \in \R^n \setminus \{0\}} \ker L(\xi)=\R \id.
\] 
\end{proposition}
\begin{proof}
If for every \(\xi \in \R^n\), \(L(\xi)[e]=0\),
then for every \(\xi \in \R^n\), there exists \(\lambda \in \R \setminus\{0\}\) such that \(e(\xi)=\lambda \xi\).
Since \(e\) is linear, there exists \(\lambda \in \R\) such that 
\(e=\lambda \id\).
\end{proof}

By the application of proposition~\ref{propositionPartialCocanceling}, we deduce

\begin{corollary}[M.\thinspace Briane and J.\thinspace Casado-Diaz, 2010 \cite{BC}]
If \(f \in L^1(\R^n; \Lin(\R^n; \R^n))\) and \(L(D)f=0\), then 
\(
 f-(\tr f)\id \in \dot{W}^{-1, \frac{n}{n-1}}\bigl(\R^n; \Lin(\R^n; \R^n)\bigr)
\)
and
\[
 \norm{f-(\tr f)\id}_{\dot{W}^{-1, n/(n-1)}}\le C\norm{f-(\tr f)\id}_{L^1}.
\]
\end{corollary}

This result is used in the study of some Navier--Stokes equation \cite{BC}.

\subsection{Proof of the Sobolev estimate}
We now have the proof of the sufficiency part of theorem~\ref{theoremPartiallyCancelingKernel}. We shall prove a quantitative version

\begin{proposition}
\label{propositionPartiallyCanceling}
Let \(n \ge 2\) and let \(A(D)\) be an elliptic linear homogeneous differential operator on \(\R^n\) from \(V\) to \(E\) and let \(P \in \Lin(E; E)\) be a projector on
\[
 \bigcap_{\xi \in \R^n \setminus\{0\}} A(\xi)[V].
\]
For every \(u \in C^\infty_c(\R^n; E)\), one has
\[
  \norm{D^{k-1} u}_{L^{\frac{n}{n-1}}} \le C\bigl(\norm{(\id - P) \circ A(D)u}_{L^1}+\norm{P \circ A(D)u}_{\dot{W}^{-1, n/(n-1)}}\bigr).
\]
\end{proposition}

The interpretation is that the image of \(A(D)\) has some bad directions \(\bigcap_{\xi \in \R^n \setminus \{0\}} A(\xi)[V]\). If one has some better control in these directions, one can have a control on \(\norm{D^{k-1}u}_{L^{n/(n-1)}}\).

\begin{proof}[Proof of proposition~\ref{propositionPartiallyCanceling}]
If \(L(D)\) is given by proposition~\ref{propositionConstructionL}, one has
\[
  \bigcap_{\xi \in \R^n \setminus\{0\}} \ker L(\xi)=\bigcap_{\xi \in \R^n \setminus \{0\}} A(\xi)[V]=P(E)=\ker(\id - P).
\]
In view of proposition~\ref{propositionPartialCocanceling}, one has 
\[
 \norm{(\id-P) \circ A(D)u}_{\dot{W}^{-1, n/(n-1)}} \le C\norm{(\id-P) \circ A(D)u}_{L^1}.
\]
Hence, 
\[
  \norm{A(D)u}_{\dot{W}^{-1, n/(n-1)}} \le C\bigl(\norm{(\id-P) \circ A(D)u}_{L^1}+\norm{P \circ A(D)u}_{\dot{W}^{-1, n/(n-1)}}\bigr).
\]
One concludes by using the ellipticity of \(A(D)\) as in the proof of proposition~\ref{propositionSufficient} that
\[
  \norm{D^{k-1} u}_{L^{n/(n-1)}} \le C'\bigl(\norm{(\id-P) \circ A(D)u}_{L^1}+\norm{P \circ A(D)u}_{\dot{W}^{-1, n/(n-1)}}\bigr).\qedhere
\]
\end{proof}

\subsection{The necessity condition for the estimate}
We finally sketch the proof of the necessity part of theorem~\ref{theoremPartiallyCancelingKernel}

\begin{proposition}
\label{propositionPartiallyCancelingKernelNecessary}
Let \(n \ge 2\) and let \(A(D)\) be an elliptic linear homogeneous differential operator on \(\R^n\) from \(V\) to \(E\) and let \(T \in \Lin(E; E)\).
If for every \(u \in C^\infty_c(\R^n; E)\) such that \(T \circ A(D) u = 0\)
\[
  \norm{D^{k-1} u}_{L^{\frac{n}{n-1}}} \le C\norm{A(D)u}_{L^1}
\]
then
\[
 \bigcap_{\xi \in \R^n \setminus\{0\}} A(\xi)[V] \cap \ker T = \{0\}.
\]
\end{proposition}
\begin{proof}
The proof follows the proof of proposition~\ref{propositionCancelingNecessary}. One chooses \(e \in \bigcap_{\xi \in \R^n \setminus\{0\}} A(\xi)[V] \cap \ker T\) and one checks that by construction of \(u_\lambda\), \(T \circ A(D) u_\lambda = 0\).
\end{proof}

\subsection{An example of partially canceling operator}
We consider the Hodge--Sobolev inequality in the case that was not treated corresponding to \eqref{HodgeSobolevL1R3}

\begin{proposition}
\label{propositionPartiallyCancelingHodge}
Let \(n \ge 2\) and \(A(D)=(d, d^*)\) be the homogeneous linear differential operator of order \(1\) on \(\R^n\) from \(\bigwedge^1 \R^n\) to \(\bigwedge^{2}\R^n \times \bigwedge^{0} \R^n\) such that for every \(\xi \in \R^n\) and \(v \in \bigwedge^1 \R^n\)
\[
  A(\xi)[v]=\bigl(\xi \wedge v, * (\xi \wedge * v)\bigr).
\]
The operator \(A(D)\) is elliptic.\\
One has
\[
 \bigcap_{\xi \in \R^n \setminus \{0\}} A(\xi)\bigl[{\textstyle\bigwedge^1} \R^n\bigr] = \{0\} \times {\textstyle\bigwedge^0} \R^n.
\]
\end{proposition}

By theorem~\ref{theoremPartiallyCancelingKernel}, we have the inequality obtained by J.\thinspace Bourgain and H.\thinspace Brezis \citelist{\cite{BB2004}*{theorem 2}\cite{BB2007}*{corollary 12}\cite{LS2005}*{main theorem (b)}\cite{VS2004Divf}*{theorem 1.1}}: for every \(u \in C^\infty_c(\R^n)\) with \(d^*u=0\),
\[
 \norm{u}_{L^{n/(n-1)}} \le C \norm{du}_{L^1}.
\]
If we use the quantitative version of of proposition~\ref{propositionPartiallyCanceling}, this gives
\begin{corollary}
\label{corollaryHodgeSobolevCritical}
For every \(u \in C^\infty_c(\R^n; \bigwedge^1\R^n)\), one has
\[
  \norm{u}_{L^{n/(n-1)}} \le C \bigl(\norm{du}_{L^1} + \norm{d^*u}_{\dot{W}^{-1, n/(n-1)}}\bigr).
\]
\end{corollary}
By the embedding of the real Hardy space \(\mathcal{H}^1(\R^n)\) in \(\dot{W}^{-1, n/(n-1)}(\R^n)\), corollary~\ref{corollaryHodgeSobolevCritical} also implies the estimate of L.\thinspace Lanzani and E.\thinspace Stein \cite{LS2005}*{main theorem (b)}
\[
  \norm{u}_{L^{n/(n-1)}} \le C \bigl(\norm{du}_{L^1} + \norm{d^*u}_{\mathcal{H}^1}\bigr).
\]

\section{Fractional and Lorentz estimates}

\label{sectionFractional}

\subsection{Sobolev estimates in fractional and Lorentz spaces}

If \(A(D)\) is a homogeneous linear differential operator of order \(k\) on \(\R^n\) from \(V\) to \(E\), one has the inequality
\[
 \norm{D^{k-1}u}_{L^\frac{n}{n-1}} \le C \norm{A(D)u}_{L^1}.
\]
This estimate can be improved in various fractional cases.

\subsubsection{Sobolev--Slobodecki\u \i{} spaces}
In the case of fractional Sobolev--Slobodecki\u \i{} spaces, we have

\begin{theorem}
\label{theoremOrderkFractional}
Let \(n \ge 1\) and let \(A(D)\) be a homogeneous linear differential operator of order \(k\) on \(\R^n\) from \(V\) to \(E\) and let \(s \in (0, 1)\) and \(p \in (1, \infty)\) be such that 
\(\frac{1}{p}-\frac{s}{n}=1-\frac{1}{n}\).
The estimate
\[
  \norm{D^{k-1}u }_{\dot{W}^{s, p}} \le C\norm{A(D) u}_{L^1},
\]
holds for every \( u \in C^\infty_c(\R^n; V)\)
if and only if \(A(D)\) is elliptic and canceling.
\end{theorem}

Here, \(\norm{v}_{\dot{W}^{s, p}}\) is the homogeneous fractional Sobolev--Slobodecki\u \i{} semi-norm, that is
\[
  \norm{v}^p_{\dot{W}^{s, p}}=\int_{\R^n} \int_{\R^n} \frac{\abs{v(x)-v(y)}^p}{\abs{x-y}^{n+sp}}\,dy\,dx.
\]
The sufficiency part of theorem~\ref{theoremOrderkFractional} is not a consequence of theorem~\ref{theoremOrderk}.

Recall that the derivative operator is canceling if and only if \(n \ge 2\) (proposition~\ref{propositionGradientCanceling}). This allows us to recover the classical result \citelist{\cite{BourgainBrezisMironescu}*{appendix D}\cite{SchmittWinkler}*{proposition 4}}

\begin{corollary}
\label{corollarynne2}
Let \(n \ge 1\), \(s \in (0, 1)\) and \(p \in (1, \infty)\) be such that 
\(
 \frac{1}{p}-\frac{s}{n}=1-\frac{1}{n}.
\)
The estimate
\[
 \norm{u}_{\dot{W}^{s, p}} \le C \norm{D u}_{L^1}
\]
holds for every \(u \in C^\infty_c(\R^n)\) if and only if \(n \ge 2\).
\end{corollary}

The sufficiency part of corollary~\ref{corollarynne2} also follows from the inequality
\[
 \norm{u}_{B^{s, p}_1} \le C \norm{\nabla u}_{L^1}
\]
for every \(u \in C^\infty_c(\R^n)\)
obtained in the more general context of anisotropic Sobolev spaces \cite{Solonnikov1975}*{theorem 2} by V.\thinspace I.\thinspace Kolyada \cite{Kolyada}*{theorem 4} or the estimate
\[
 \norm{u}_{B^{s, p}_1} \le C \norm{\nabla u}_{L^1}^s \norm{u}_{L^\frac{d}{d-1}}^{1-s}
\]
obtained by A.\thinspace Cohen, W.\thinspace Dahmen, I.\thinspace Daubechies and R.\thinspace DeVore \cite{CDDD}*{theorem 1.4} (see also J.\thinspace Bourgain, H.\thinspace Brezis and P.\thinspace Mironescu \cite{BourgainBrezisMironescu}*{lemma D.2})
together with standard embeddings between Besov spaces and the identification of Besov spaces with fractional Sobolev--Slobo\-decki\u \i{} spaces \cite{Triebel}*{2.3.2(5), 2.3.5(3) and 2.5.7(9)}. A counterexample when \(n=1\) can be obtained by taking regularizations of a characteristic function \cite{SchmittWinkler}.

\subsubsection{Triebel--Lizorkin spaces}

Theorem \ref{theoremOrderk} also extends to Triebel--Lizorkin spaces, as it was already the case for the Hodge--Sobolev inequality \eqref{HodgeSobolevL1} \cite{VS2010}*{theorem 1}. In the scale of Triebel--Lizorkin spaces, we have

\begin{theorem}
\label{theoremTriebelLizorkin}
Let \(n \ge 1\) and let \(A(D)\) be a homogeneous linear differential operator of order \(k\) on \(\R^n\) from \(V\) to \(E\), let \(s \in (k-\frac{n}{n-1}, k)\) and \(p \in (1, \infty)\) be such that 
\(
 \frac{1}{p}-\frac{s}{n}=1-\frac{k}{n}
\),
and let \(q \in (0, \infty]\).
The estimate 
\[
  \norm{u}_{\dot{F}^{s}_{p, q}} \le C\norm{A(D) u}_{L^1},
\]
holds for every \( u \in C^\infty_c(\R^n; V)\) if and only if \(A(D)\) is elliptic and canceling.
\end{theorem}

We need the restriction \(s > k-\frac{n}{n-1}\) to prove the ellipticity. As discussed at the end of section~\ref{sectionNecessaryEllipticity}, the theorem fails for \(s \le k-2\). This raises the problem

\begin{openproblem}
Let \(n \ge 3\). Does theorem~\ref{theoremTriebelLizorkin} fail for \(s \in (k-2, k-\frac{n}{n-1}]\)?
\end{openproblem}

\subsubsection{Besov spaces}
The extension of the Hodge--Sobolev inequality in Besov spaces \citelist{\cite{MM2009}*{proposition 1}\cite{VS2010}*{theorem 1}} to homogeneous linear differential operators is

\begin{theorem}
\label{theoremBesov}
Let \(n \ge 1\) and let \(A(D)\) be a homogeneous linear differential operator of order \(k\) on \(\R^n\) from \(V\) to \(E\), let \(s \in (k-\frac{n}{n-1}, k)\) and \(p \in (1, \infty)\) be such that 
\(
 \frac{1}{p}-\frac{s}{n}=1-\frac{k}{n}
\),
and let \(q \in (1, \infty)\).
The estimate 
\[
  \norm{u }_{\dot{B}^{s}_{p, q}} \le C\norm{A(D) u}_{L^1},
\]
holds for every \( u \in C^\infty_c(\R^n; V)\) if and only if \(A(D)\) is elliptic and canceling.
\end{theorem}

In the case \(q=\infty\), the ellipticity alone is necessary and sufficient (see proposition~\ref{propositionLimitingBesov}). When \(q=1\), the ellipticity and the cancellation are necessary, but as for the Hodge--Sobolev estimate \cite{VS2010}*{open problem 1} we do not know whether they are sufficient:

\begin{openproblem}
\label{problemBesov}
Let \(k \ge n\) and \(A(D)\) be a homogeneous linear differential operator of order \(k\) on \(\R^n\) from \(V\) to \(E\). Assume that \(A(D)\) is elliptic and canceling and that \(s \in (k-n, n)\) and \( p \in (1, \infty)\) satisfy \(\frac{1}{p}-\frac{s}{n}=1-\frac{k}{n}\). Does one have for every \(u \in C^\infty_c(\R^n; V)\), 
\[
  \norm{u}_{\dot{B}^s_{p, 1}} \le C \norm{A(D)u}_{L^1}?
\]
\end{openproblem}

The answer is positive in the scalar case \(V=\R\) \cite{Kolyada}*{corollary 1}.
The question is already open for the Hodge--Sobolev inequality \cite{VS2010}*{open problem 1}.

\subsubsection{Lorentz spaces}
Finally, in the framework of Lorentz spaces, we have, as for the Hodge--Sobolev estimate \cite{VS2010}*{theorem 3}
\begin{theorem}
\label{theoremLorentz}
Let \(n \ge 2\) and let \(A(D)\) be a homogeneous linear differential operator of degree \(k\) on \(\R^n\) from \(V\) to \(E\)  and \(q \in (1, \infty) \).
The estimate 
\[
  \norm{D^{k-1}u }_{L^{n/(n-1), q}} \le C\norm{A(D) u}_{L^1},
\]
holds for every \( u \in C^\infty_c(\R^n; V)\) if and only if \(A(D)\) is elliptic and canceling.
\end{theorem}

Again, when \(q=\infty\), the ellipticity alone is necessary and sufficient (see proposition~\ref{propositionLimitingLorentz}). If \(q=1\), the ellipticity and the cancellation are necessary, but as for the Hodge--Sobolev estimate \cite{VS2010}*{open problem 2} we could not determine whether they are sufficient

\begin{openproblem}
\label{problemLorentz}
Let \(k \ge n\) and \(A(D)\) be a homogeneous linear differential operator of order \(k\) on \(\R^n\) from \(V\) to \(E\). Assume that \(A(D)\) is elliptic and canceling. Does one have for every \(u \in C^\infty_c(\R^n; V)\), 
\[
  \norm{D^{k-1}u}_{L^{\frac{n}{n-1}, 1}} \le C \norm{A(D)u}_{L^1}?
\]
\end{openproblem}

This property is true when one considers the gradient in Sobolev spaces for Lorentz spaces \cite{Alvino1977}.

\medbreak

Since
\[
 \norm{D^{k-n}u}_{L^\infty} \le C\norm{I_{n-(k-s)}}_{\dot{B}^{s}_{\frac{n}{k-s}, \infty}}\norm{u}_{\dot{B}^{s}_{\frac{n}{n-(k-s)}, 1}}
\]
and 
\[
  \norm{D^{k-n}u}_{L^\infty} \le C\norm{I_{n-(k-\ell)}}_{L^{\frac{n}{k-\ell}, \infty}}\norm{D^\ell u}_{L^{\frac{n}{n-(k-\ell)}, 1}},
\]
where \(I_\alpha\) is the Riesz potential of order \(\alpha \in (0, n)\) defined for \(x \in \R^n \setminus \{0\}\) by  \(I_\alpha(x)=\frac{\pi^{n/2}2^\alpha\Gamma(\alpha/2)}{\Gamma((n-\alpha)/2)\abs{x}^{n-\alpha}}\),
a positive answer to either open problem~\ref{problemBesov} or open problem~\ref{problemLorentz} would imply the estimate
\[
  \norm{D^{k-n} u}_{L^\infty} \le C \norm{A(D)u}_{L^1}.
\]
This motivates the problem

\begin{openproblem}
\label{problemLinfty}
Let \(k \ge n\) and \(A(D)\) be a homogeneous differential operator of order \(k\) on \(\R^n\) from \(V\) to \(E\). Assume that \(A(D)\) is elliptic and canceling. Does one have for every \(u \in C^\infty_c(\R^n; V)\), 
\begin{equation}
\label{ineqLinfty}
  \norm{D^{k-n} u}_{L^\infty} \le C \norm{A(D)u}_{L^1}?
\end{equation}
\end{openproblem}

The answer is positive in the scalar case: for every \(u \in C^\infty_c(\R^n)\), 
\[
  \norm{u}_{L^\infty} \le \norm{D^n u}_{L^1}.  
\]
A nontrivial vector example is given by the estimate 
\begin{equation}
\label{ineqLinftynablau}
 \norm{\nabla u}_{L^\infty} \le C \norm{\Delta \nabla u}_{L^1}
\end{equation}
for every \(u \in C^\infty_c(\R^2)\). This estimate was obtained by J.\thinspace Bourgain and H.\thinspace Brezis \citelist{\cite{BB2004}*{remark 5}\cite{BB2007}*{theorem 3}} (see also \citelist{\cite{VS2006BMO}*{corollary 4.9}\cite{BVS2007}*{theorem 2.1}}). For an alternative proof, note that
\begin{equation}
\begin{split}
 \nabla u &= (\nabla \Div) \Delta^{-2}  (\Delta \nabla u),\\
          &= \left(\begin{smallmatrix}
                \partial_1^2-\partial_2^2 & 2\partial_1 \partial_2 \\
                2\partial_1 \partial_2 & \partial_2^2-\partial_1^2 \\
             \end{smallmatrix}\right)
\Delta^{-2} (\Delta \nabla u).
\end{split}
\end{equation}
If \(G\) denotes the fundamental solution of \(\Delta^2\) in \(\R^2\), P.\thinspace Mironescu has shown that \(\partial_1^2G-\partial_2^2G\) and \(\partial_1 \partial_2 G\) are bounded \cite{Mironescu2010}*{proposition 1}. The estimate \eqref{ineqLinftynablau} then follows.

More generally, if \(n\) is even, one has 
\[
\begin{split}
 \nabla u &= (\nabla \Div) \Delta^{-\frac{n}{2}-1}  (\Delta^\frac{n}{2} \nabla u),\\
          &= \frac{1}{n-1} (n\nabla \Div -\Delta) \Delta^{-\frac{n}{2}-1}  (\Delta^\frac{n}{2} \nabla u).
\end{split}
\]
If \(G\) denotes the Green function of \(\Delta^{\frac{n}{2}+1}\) on \(\R^n\),  \(nD^2G-\Delta G \id \in L^\infty\) \cite{Mironescu2010}*{proposition 3}, and therefore 
\[
   \norm{u}_{L^\infty} \le C \norm{\Delta^\frac{n}{2} \nabla u}_{L^1}. 
\]
Also note that as noticed in remark~\ref{remarkCancelingLinfty}, canceling is not necessary for \eqref{ineqLinfty}.

\subsection{$L^1$ estimates and cocanceling operators}
In order to prove the fractional and Lorentz space estimates, we first extend the results of section~\ref{sectionCocanceling} concerning cocanceling operators

\begin{proposition}
\label{propositionVectorL1Fract}
Let \(L(D)\) be a homogeneous differential operator from \(E\) to \(F\), let \(s \in (0, 1)\) and \(p \in (1, \infty)\) be such that \(sp=n\). If \(L(D)\) is cocanceling,  \(f \in L^1(\R^n; E)\) and \(L(D)f=0\) in the sense of distributions, then for every \(\varphi \in C^\infty_c(\R^n; E)\), 
\[
  \int_{\R^n} f \cdot \varphi \le C \norm{f}_{L^1}\norm{\varphi}_{\dot{W}^{s,p}}.
\]
\end{proposition}

\begin{proof}
The proof is similar to the proof of proposition~\ref{propositionVectorL1Fract}, it relies on the counterpart of proposition~\ref{propVSHigherOrder} for fractional Sobolev--Slobodecki\u \i{} spaces~\cite{VS2008}*{(4)}.
\end{proof}

The cocancellation condition is here necessary (see the proof of proposition~\ref{propositionCocancelingNecessary}).

\medbreak

One can also use the same kind of arguments in order to obtain a counterpart of proposition~\ref{propVSHigherOrder} for Triebel--Lizorkin spaces \cite{VS2010}*{proof of proposition 2.1}. This shows that one can replace in the statement of proposition~\ref{propositionVectorL1Fract} \(\dot{W}^{s, p}(\R^n; E)\) by \(\dot{F}^{s}_{p, q}(\R^n; E)\) for every \(q \ge  1\). This can also be deduced from proposition~\ref{propositionVectorL1Fract} by standard embeddings between fractional spaces \cite{Triebel}*{theorem 2.7.1 and \S 5.2.5}:

\begin{proposition}
\label{propositionCocancelingTriebel}
Let \(L(D)\) be a homogeneous differential operator from \(E\) to \(F\), let \(s \in (0, 1)\) and \(p \in (1, \infty)\) be such that \(sp=n\) and let \(q \in [1, \infty]\). If \(L(D)\) is cocanceling,  \(f \in L^1(\R^n; E)\) and \(L(D)f=0\) in the sense of distributions, then for every \(\varphi \in C^\infty_c(\R^n; E)\), 
\[
  \int_{\R^n} f \cdot \varphi \le C \norm{f}_{L^1}\norm{\varphi}_{\dot{F}^s_{p,q}}.
\]
\end{proposition}
The cocancellation condition is still necessary for Triebel--Lizorkin spaces.

\medbreak

For Besov spaces, one has
\begin{proposition}
\label{propositionCocancelingBesov}
Let \(L(D)\) be a homogeneous differential operator on \(\R^n\) from \(E\) to \(F\), let \(s \in (0, 1)\) and \(p \in (1, \infty)\) be such that \(sp=n\) and let \(q \in (1, \infty]\). If \(L(D)\) is cocanceling,  \(f \in L^1(\R^n; E)\) and \(L(D)f=0\) in the sense of distributions, then for every \(\varphi \in C^\infty_c(\R^n; E)\), 
\[
  \int_{\R^n} f \cdot \varphi \le C \norm{f}_{L^1}\norm{\varphi}_{\dot{B}^s_{p,q}}.
\]
\end{proposition}

This proposition is deduced from proposition~\ref{propositionVectorL1Fract} or from proposition~\ref{propositionCocancelingTriebel}.
The case \(q=1\) is a consequence of the estimate 
\[
 \norm{\varphi}_{L^\infty} \le C\norm{\varphi}_{\dot{B}^s_{p,1}},
\]
the cocancellation condition is not necessary in this case (see proposition~\ref{propositionEllipticityBesov}). In the other cases, it is necessary.

The case \(q=\infty\) is open. The current arguments fail in this case because proposition~\ref{propVSHigherOrder} relies on a Fubini-type property that is only present in Triebel--Lizorkin spaces. Proposition~\ref{propVSHigherOrder} can thus only be proved in those spaces; the Nikol'ski\u\i{} spaces \(B^s_{p, \infty}\) do not embed in this scale of spaces.

We remark that a counterexample cannot be constructed by taking for \(\varphi\) a regularization of \(x \in \R^n \mapsto \log \abs{x}\)
\begin{proposition}
Let \(L(D)\) be a homogeneous differential operator of order \(k\) on \(\R^n\) from \(E\) to \(F\). If \(L(D)\) is cocanceling,  \(f \in L^1(\R^n; E)\) and \(L(D)f=0\) in the sense of distributions, then for every \(\varphi \in C^\infty_c(\R^n; E)\),  
\[
\int_{\R^n} f \cdot \varphi
\le C \norm{f}_{L^1}\sum_{\ell=1}^k \sup_{x \in \R^n}\abs{x}^\ell \abs{D^\ell \varphi(x)}.
\]
\end{proposition}

\begin{proof}
We extend the argument proposed in the case where \(L(D)\) is the divergence operator \cite{VS2006BMO}*{proposition 4.3}. 
Let \(K_\alpha\) be given by lemma~\ref{lemmaCharKalpha} and define \(P : \R^n \mapsto \Lin(E; F)\) for \(x \in \R^n\) by
\[
  P(x)=\sum_{\substack{\alpha \in \N^n\\ \abs{\alpha}=k}} \frac{x^\alpha}{\alpha!} K_\alpha^*.
\]
One has in view of \eqref{eqCharKalpha}, for every \(x \in \R^n\), 
\[
  \bigl(L(D)^*P\bigr)(x)=\sum_{\substack{\alpha \in \N^n\\ \abs{\alpha}=k}} \tfrac{\partial^\alpha x^\alpha}{\alpha !} L_\alpha^* \circ K_\alpha^* =\id.
\]
Therefore, since \(L(D)f=0\), 
\[
  \int_{\R^n} f \cdot \varphi=\int_{\R^n} f \cdot \bigl(L(D)^*P\bigr) [\varphi]
=\int_{\R^n} f \cdot \bigl((L(D)^*P)[\varphi] - L(D)^*(P[\varphi])\bigr).
\]
One concludes by noting that for every \(x \in \R^n\), 
\[
 \bigabs{\bigl(L(D)^*P\bigr)(x) [\varphi(x)] - \bigl(L(D)^*(P[\varphi])\bigr)(x)}\le C\sum_{\ell=1}^k \abs{x}^\ell \abs{D^\ell \varphi(x)}.\qedhere
\]
\end{proof}

The estimate of proposition~\ref{propCocancelingSufficient} becomes in the framework of Lorentz spaces

\begin{proposition}
\label{propCocancelingSufficientLorentz}
Let \(L(D)\) be a homogeneous differential operator from \(E\) to \(F\) and \(q \in [1, \infty)\). If \(L(D)\) is cocanceling,  \(f \in L^1(\R^n; E)\) and \(L(D)f=0\) in the sense of distributions, then for every \(\varphi \in C^\infty_c(\R^n; E)\), 
\[
  \int_{\R^n} f \cdot \varphi \le C \norm{f}_{L^1}\norm{D \varphi}_{L^{n, q}}.
\]
\end{proposition}

Again the cocancellation condition is necessary if \(q > 1\) and the case \(q=\infty\) is open.

\subsection{Proofs of the Sobolev estimates}
The proof of the Sobolev estimates in the fractional and Lorentz spaces can be done as in section~\ref{sectionSobolev}.
First one note that the results in section~\ref{subsectionClassicalElliptic} extend to fractional Sobolev--Slobodecki\u \i{} spaces, Triebel--Lizorkin, Besov and Lorentz--Sobolev spaces by standard multiplier theorems adapted to these spaces \cite{Triebel}*{theorem 2.3.7}.

Our previous approach extends to fractional Sobolev--Slobodecki\u \i{} spaces: by using proposition~\ref{propositionVectorL1Fract} instead of \eqref{theoremCocanceling} and the counterpart of proposition~\ref{propCaldZygmundWk1p} in fractional Sobolev--Slobodecki\u \i{} spaces, we obtain the sufficiency part of theorem~\ref{theoremOrderkFractional}

\begin{proposition}
\label{propositionSufficientOrderkFractional}
Let \(A(D)\) be a homogeneous linear differential operator of order \(k\) on \(\R^n\) from \(V\) to \(E\) and let \(s \in (0, 1)\) and \(p \in (1, \infty)\) be such that 
\(
 \frac{1}{p}-\frac{s}{n}=1-\frac{1}{n}
\).
If \(A(D)\) is elliptic and canceling, then for every \( u \in C^\infty_c(\R^n; V)\),
\[
  \norm{D^{k-1}u }_{\dot{W}^{s, p}} \le C\norm{A(D) u}_{L^1}.
\]
\end{proposition}

Similarly, in Triebel--Lizorkin spaces, one has the sufficiency part of theorem~\ref{theoremTriebelLizorkin}:

\begin{proposition}
\label{propositionSufficientTriebelLizorkin}
Let \(A(D)\) be a homogeneous linear differential operator of order \(k\) on \(\R^n\) from \(V\) to \(E\), let \(s \in (k-n, k)\) and \(p \in (1, \infty)\) be such that 
\(
 \frac{1}{p}-\frac{s}{n}=1-\frac{k}{n},
\)
and let \(q \in [1, \infty]\).
If \(A(D)\) is elliptic and canceling, then for every \( u \in C^\infty_c(\R^n; V)\),
\[
  \norm{u}_{\dot{F}^{s}_{p, q}} \le C\norm{A(D) u}_{L^1}.
\]
\end{proposition}

\begin{proof}
For \(q > 1\), the proof goes as the proof of proposition~\ref{propositionSufficient}, using proposition~\ref{propositionCocancelingTriebel} instead of theorem~\ref{theoremCocanceling} and the counterpart of proposition~\ref{propCaldZygmundWk1p} in Triebel--Lizorkin spaces. One can then treat the case \(q \in (0, 1]\) by embeddings between Triebel--Lizorkin spaces \citelist{\cite{Triebel}*{theorem 2.7.1}\cite{RunstSickel}*{proposition 2.2.3}}: if \(t \in (s, n)\) and \(r \in (1, \infty )\) are such that  \(\frac{1}{r}-\frac{u}{n}=1-\frac{k}{n}\) and \(u \in (0, \infty]\), 
then 
\[
\label{eqTriebelEmbedding}
 \norm{u}_{\dot{F}^{s}_{p, q}} \le C\norm{u}_{\dot{F}^{t}_{r, u}}.\qedhere
\]
\end{proof}

In the case of the Besov spaces, one has the sufficiency part of theorem~\ref{theoremBesov}

\begin{proposition}
\label{propositionSufficientBesov}
Let \(A(D)\) be a homogeneous linear differential operator of order \(k\) on \(\R^n\) from \(V\) to \(E\), let \(s \in (k-n, k)\) and \(p \in (1, \infty)\) be such that 
\(
 \frac{1}{p}-\frac{s}{n}=1-\frac{k}{n},
\)
and let \(q \in (1, \infty]\).
If \(A(D)\) is elliptic and canceling, then for every \( u \in C^\infty_c(\R^n; V)\),
\[
  \norm{u }_{\dot{B}^{s}_{p, q}} \le \norm{A(D) u}_{L^1}.
\]
\end{proposition}

When \(q=\infty\), the ellipticity alone is sufficient (proposition~\ref{propositionLimitingBesov}). 
The proof is similar to the proof of proposition~\ref{propositionSufficientTriebelLizorkin}, except that the counterpart of \eqref{eqTriebelEmbedding} only holds if \(u \le q\). 

\medbreak

Finally, we have in Lorentz spaces

\begin{proposition}
\label{propositionLorentz}
Let \(A(D)\) be a homogeneous linear differential operator of degree \(k\) on \(\R^n\) from \(V\) to \(E\)  and \(q \in (1, \infty) \).
If \(A(D)\) is elliptic and canceling, then for every \( u \in C^\infty_c(\R^n; V)\),
\[
  \norm{D^{k-1}u }_{L^{n/(n-1), q}} \le \norm{A(D) u}_{L^1},
\]
\end{proposition}

\subsection{Necessity of the ellipticity}

The proof of proposition~\ref{propositionEllipticitySobolev} applies to fractional spaces and yields
\begin{proposition}
\label{propositionEllipticitySobolevFractional}
Let \(A(D)\) be a homogeneous differential operator on \(\R^n\)  of order \(k\) from \(V\) to \(E\).
Let \(s \in (k-\frac{n}{n-1}, k)\), \(p \ge  1\) and \(q > 1\) be such that 
\(
  \frac{1}{q}-\frac{s}{n}=\frac{1}{p}-\frac{k}{n}
\).
If for every \(u \in C^\infty_c(\R^n; V)\), 
\[
 \norm{u}_{\dot{W}^{s, q}} \le C \norm{A(D) u}_{L^p},
\]
then \(A(D)\) is elliptic.
\end{proposition}

For \( s \in [k-1, k)\), this is a consequence of corollary~\ref{corollaryEllipticitySobolev} by classical embeddings theorems for fractional Sobolev--Slobodecki\u \i{} spaces \cite{Adams}*{theorem 7.57}.

\begin{proof}
One begins as the in proof of proposition~\ref{propositionEllipticitySobolev}. One notes then that
\[
  \norm{u_\lambda}_{\dot{W}^{s, q}}=C \lambda^{n-1} \abs{v}+o(\lambda^{n-1}). \qedhere
\]
\end{proof}

In Triebel--Lizorkin spaces, one has

\begin{proposition}
\label{propositionEllipticityLizorkinTriebel}
Let \(A(D)\) be a homogeneous differential operator on \(\R^n\)  of order \(k\) from \(V\) to \(E\).
Let \(s \in (k-\frac{n}{n-1}, k)\), \(p \ge  1\), \( r > 0\) and \(q > 1\) be such that 
\(
  \frac{1}{q}-\frac{s}{n}=\frac{1}{p}-\frac{k}{n}
\).
If for every \(u \in C^\infty_c(\R^n; V)\), 
\[
 \norm{u}_{\dot{F}^{s}_{q, r}} \le C \norm{A(D) u}_{L^p},
\]
then \(A(D)\) is elliptic.
\end{proposition}

Proposition~\ref{propositionEllipticityLizorkinTriebel} can be obtained either by a direct proof or by deduction from proposition~\ref{propositionEllipticitySobolevFractional} by standard embedding theorems and the characterization of fractional Sobolev--Slobodecki\u \i{} spaces as Triebel--Lizorkin spaces \cite{Triebel}*{theorems 2.5.7 and 2.7.1}.

\medbreak

For Besov spaces we have

\begin{proposition}
\label{propositionEllipticityBesov}
Let \(A(D)\) be a homogeneous differential operator on \(\R^n\)  of order \(k\) from \(V\) to \(E\).
Let \(s \in (k-\frac{n}{n-1}, k)\), \(p \ge  1\), \( r > 0\) and \(q > 1\) be such that 
\(
  \frac{1}{q}-\frac{s}{n}=\frac{1}{p}-\frac{k}{n}
\).
If for every \(u \in C^\infty_c(\R^n; V)\), 
\[
 \norm{u}_{\dot{B}^{s}_{q, r}} \le C \norm{A(D) u}_{L^p},
\]
then \(A(D)\) is elliptic.
\end{proposition}

Proposition \ref{propositionEllipticityBesov} cannot be deduced from proposition~\ref{propositionEllipticitySobolev}. Such an argument would in fact impose the additional restriction that \( r \le q\) that does not appear with the direct argument.

\medbreak

Finally, for Lorentz spaces, one has

\begin{proposition}
\label{propositionEllipticityLorentz}
Let \(A(D)\) be a homogeneous differential operator on \(\R^n\)  of order \(k\) from \(V\) to \(E\).
Let  \(q > 1\) be such that 
\(
  \frac{1}{q}-\frac{s}{n}=\frac{1}{p}-\frac{k}{n}
\).
If for every \(u \in C^\infty_c(\R^n; V)\), 
\[
 \norm{D^{k-1} u}_{L^{q, r}} \le C \norm{A(D) u}_{L^p},
\]
then \(A(D)\) is elliptic.
\end{proposition}

When \(r \le q\), this is an immediate consequence of proposition~\ref{propositionEllipticitySobolev}.
When \( r > q\), the proof of proposition~\ref{propositionEllipticitySobolev} applies and gives the conclusion.

\subsection{Necessity of the cancellation}

Concerning fractional spaces, the proof of proposition~\ref{propositionCancelingNecessary} allows to prove

\begin{proposition}\label{propositionCancelingNecessaryFractionalSobolev}
Let \(A(D)\) be an elliptic homogeneous linear differential operator of order \(k\) on \(\R^n\) from \(V\) to \(E\), let \(s \in (0, 1)\), \(p \ge 1\) and \(\ell \in \{1, \dotsc, n-1\}\) such that \( \ell \ge n-k\) and
\(
  \frac{1}{p}-\frac{\ell+s}{n}=1-\frac{1}{n}
\).
If for every \(u \in C^\infty_c(\R^n; V)\), 
\[
 \norm{D^\ell u}_{\dot{W}^{s,p}} \le \norm{A(D)u}_{L^1},
\]
then \(A(D)\) is canceling.
\end{proposition}
\begin{proof}
One proceeds as in the proof of proposition~\ref{propositionCancelingNecessary}, using the fact that if \eqref{equalphahomog} is satisfied, then \(u^\alpha\) does not have finite fractional Sobolev--Slobodecki\u \i{} norm and applying the Fatou property in fractional Sobolev--Slobodecki\u \i{} spaces: if \( \partial^\alpha u_\lambda \to u^\alpha \) almost everywhere as \(\lambda \to \infty\), then 
\[
  \int_{\R^n} \int_{\R^n} \frac{\abs{u^\alpha(x)-u^\alpha(y)}}{\abs{x-y}^{n+sp}} \,dx\,dy
\le \liminf_{\lambda \to \infty} \int_{\R^n} \int_{\R^n} \frac{\abs{\partial^\alpha u_\lambda(x)-\partial^\alpha u_\lambda(y)}}{\abs{x-y}^{n+sp}} \,dx\,dy.\qedhere
\]
\end{proof}

\begin{proposition}\label{propositionCancelingNecessaryTriebelLizorkin}
Let \(A(D)\) be an elliptic homogeneous linear differential operator of order \(k\) on \(\R^n\) from \(V\) to \(E\), let \(p \in (1, \infty)\) and \(s \in (k-n, k)\) be such that 
\(
 \frac{1}{p}-\frac{s}{n}=1-\frac{k}{n}
\)
and let \( q \in (0, \infty]\). If for every \(u \in C^\infty_c(\R^n; V)\), 
\[
 \norm{u}_{\dot{F}^s_{p, q}} \le C\norm{A(D)u}_{L^1},
\]
then \(A(D)\) is canceling.
\end{proposition}

When \(s \ge 0\), this is a consequence of proposition~\ref{propositionCancelingNecessaryFractionalSobolev},  classical embeddings between Triebel--Lizorkin spaces \cite{Triebel}*{theorem 2.7.1}
and the equivalence between fractional Sobolev--Slobodecki\u \i{} spaces and Triebel--Lizorkin spaces \cite{Triebel}*{theorem 2.5.7}
\begin{proof}[Proof of proposition~\ref{propositionCancelingNecessaryTriebelLizorkin}]
Follow the proof of proposition~\ref{propositionCancelingNecessary} till \eqref{equalphahomog} with \((-\Delta)^\frac{s}{2}u\) instead of \(\partial^\alpha u\). Define
\[
  u^s(x)=\lim_{\lambda \to \infty} (-\Delta)^{\frac{s}{2}} u_\lambda(x).
\]
One has in place of \eqref{equalphahomog} for each \(x \in \R^n \setminus \{0\}\) and \(t \in (0, \infty)\)
\begin{equation}
\label{equshomog}
  u^s(tx)=\frac{u^s(x)}{t^{n-(k-s)}}
\end{equation}
Therefore, \(u^s \not \in F^0_{p, q}(\R^n; V)\) if and only if \(u^s \not \equiv 0\) \cite{RunstSickel}*{lemma 2.3.1/1}.

Since 
\[
 (-\Delta)^\frac{s}{2} u_\lambda \le \frac{C}{\abs{x}^{n-(k-s)}},
\]
one has \( (-\Delta)^\frac{s}{2} u_\lambda \to u^s\) as \(\lambda \to \infty\) in \(L^1_{\mathrm{loc}}(\R^n; V)\). By the Fatou property for Triebel--Lizorkin spaces \cite{Franke}*{} (see also \cite{RunstSickel}*{proposition 2.1.3/2}), \(\norm{u_\lambda}_{\dot{F}^s_{p, q}}\) is not bounded as \(\lambda \to \infty\). One concludes as in the proof of proposition~\ref{propositionCancelingNecessary}.
\end{proof}

Similarly, one can prove in Besov spaces

\begin{proposition}\label{propositionCancelingNecessaryBesov}
Let \(A(D)\) be an elliptic homogeneous linear differential operator of order \(k\) on \(\R^n\) from \(V\) to \(E\), let \(p \in (1, \infty)\) and \(s \in (k-n, k)\) be such that 
\(
 \frac{1}{p}-\frac{s}{n}=1-\frac{k}{n}
\)
and \( q \in (0, \infty)\). If for every \(u \in C^\infty_c(\R^n; V)\), 
\[
 \norm{u}_{\dot{B}^s_{p, q}} \le C\norm{A(D)u}_{L^1},
\]
then \(A(D)\) is canceling.
\end{proposition}

The restriction \( q < \infty\) comes from the fact that \eqref{equshomog} is not incompatible with \(u_\alpha \in B^s_{p, \infty}(\R^n; V)\).
This restriction is essential as shows

\begin{proposition}
\label{propositionLimitingBesov}
Let \(A(D)\) be an elliptic linear homogeneous differential operator of order \(k\) on \(\R^n\) from \(V\) to \(E\) and let \(s \in (k-n, k)\) and \(p \in [1, \infty)\) be such that
\(
 \frac{1}{p}-\frac{s}{n}=1-\frac{k}{n}.
\)
For every \(u \in C^\infty_c(\R^n; V)\), 
\[
 \norm{u}_{\dot{B}^s_{p, \infty}} \le C \norm{A(D)u}_{L^1}.
\]
\end{proposition}
\begin{proof}
Define \(G : \R^n \setminus \{0\} \to \Lin(E; V)\) such that for every \(\xi \in \R^n \setminus \{0\}\), 
\[
  \widehat{G}(\xi)=\abs{\xi}^s \bigl(A(\xi)^* \circ  A(\xi)\bigr)^{-1}\circ A^*(\xi).
\]
Since \(\widehat{G}\) is homogeneous of degree \(-(k-s)\), \(G\) is homogeneous of degree \(-(n-(k-s))\)
and therefore \(G \in \dot{B}^0_{p, \infty}(\R^n; \Lin(V; E))\). Since \(\norm{\cdot}_{\dot{B}^s_{p, \infty}}\) is a norm, by convexity,
\[
  \norm{u}_{B^{s}_{p, \infty}} = \norm{G \ast (A(D)u)}_{B^{0}_{p, \infty}}
 \le \norm{G}_{B^{0}_{p, \infty}} \norm{A(D)u}_{L^1}.\qedhere
\]
\end{proof}

An alternative argument would be to use the estimate \citelist{\cite{SickelTriebel}*{theorem 3.1.1}\cite{RunstSickel}*{theorem 2.2.2}}
\[
 \norm{A(D)u}_{B^{0}_{1, \infty}} \le C \norm{A(D)u}_{L^1}.
\]
together with the theory of Fourier multipliers on Besov spaces \citelist{\cite{RunstSickel}*{proposition 2.1.6/5}\cite{Triebel}*{theorem 2.3.7}} and the embeddings between Besov spaces \citelist{\cite{Triebel}*{theorem 2.7.1}\cite{RunstSickel}*{theorem 2.2.3}}.

\medbreak

The argument of proposition~\ref{propositionCancelingNecessary} still applies to Lorentz space estimates

\begin{proposition}
\label{propositionCancelingNecessaryLorentz}
Let \(A(D)\) be an elliptic homogeneous linear differential operator of order \(k\) on \(\R^n\) from \(V\) to \(E\) and let \(q \in [1, \infty)\).
If for every \(u \in C^\infty_c (\R^n; V)\), 
\[
 \norm{u}_{L^{\frac{n}{n-1}, q}} \le C \norm{A(D) u}_{L^1},
\]
then \(A(D)\) is canceling.
\end{proposition}

This only follows from proposition~\ref{propositionCancelingNecessary} when \(q \le \frac{n}{n-1}\). The proof is similar to that of proposition~\ref{propositionCancelingNecessaryBesov}, using the Fatou property for Lorentz spaces, and the fact that for \( q \in [1, \infty)\), there are no nonzero homogeneous functions. 

Again the restriction \(q < \infty\) is optimal, as one has

\begin{proposition}
\label{propositionLimitingLorentz}
Let \(A(D)\) be a linear homogeneous elliptic operator of order \(k\) on \(\R^n\) from \(V\) to \(E\). For every \(u \in C^\infty_c(\R^n; V)\), 
\[
 \norm{D^{k-1}u}_{L^{\frac{n}{n-1}, \infty}} \le C \norm{A(D)u}_{L^1}.
\]
\end{proposition}
\begin{proof}
The proof is similar to the proof of proposition~\ref{propositionLimitingBesov}; an alternate proof would start from a weak \(L^1\) estimate for the elliptic operator together with Sobolev embeddings in the framework of Marcinkiewicz spaces.
\end{proof}

\section{Strong Bourgain--Brezis estimates}

\label{sectionBB}

If \(A(D)\) is a linear homogeneous differential operator of order
\(k\) on \(\R^n\), one has the estimates
\begin{equation}
\label{ineqBBL1}
 \norm{D^{k-1}u }_{L^{\frac{n}{n-1}}} \le C\norm{A(D) u}_{L^1}
\end{equation}
and
\begin{equation}
\label{ineqBBW1n}
 \norm{D^{k-1}u }_{L^{\frac{n}{n-1}}} \le C\norm{A(D) u}_{\dot{W}^{-1, n/(n-1)}}.
\end{equation}
In view of these estimates, one can wonder whether one can obtain a stronger statement using a weaker norm \(\norm{A(D) u}_{L^1+\dot{W}^{-1, n/(n-1)}}\). 

J.\thinspace Bourgain and H.\thinspace Brezis \citelist{\cite{BB2002}*{(8)}\cite{BB2003}*{lemma 1}\cite{BB2004}*{remark 6}\cite{BB2007}*{corollary 12}} have obtained such results for the gradient and the exterior derivative. Relying on their abstract results, we prove a similar counterpart of proposition~\ref{propositionSufficient} in which a weaker norm of \(A(D)u\) is taken.

\begin{theorem}
\label{theoremOrderkBB}
Let \(A(D)\) be a linear homogeneous differential operator of order
\(k\) on \(\R^n\) from \(V\) to \(E\). If \(A(D)\) is elliptic and canceling, then for every
\(u \in C^\infty_c(\R^n; V)\),
\[
  \norm{D^{k-1}u }_{L^{\frac{n}{n-1}}} \le C\norm{A(D) u}_{L^1+\dot{W}^{-1, n/(n-1)}}.
\]
\end{theorem}

These estimates are not a consequence of \eqref{ineqBBL1} and \eqref{ineqBBW1n}. Indeed, from the definition of  \(\norm{A(D) u}_{L^1+\dot{W}^{-1, n/(n-1)}}\), there exists \(f \in C^\infty_c(\R^n; E)\) such that 
\(A(D)u-f \in \dot{W}^{-1, \frac{n}{n-1}}(\R^n; E)\) and 
\[
 \norm{f}_{L^1}+\norm{A(D)u-f}_{\dot{W}^{-1, \frac{n}{n-1}}} \le 2 \norm{A(D)u}_{L^1+\dot{W}^{-1,\frac{n}{n-1}}}.
\]
but nothing says that \(f\) can be written as \(f= A(D) w\) with \(w \in C^\infty_c(\R^n; V)\) with the useful estimates.

It is not known whether theorem~\ref{theoremOrderkBB} holds in any other Sobolev space  \cite{BB2007}*{open problem 2}, that is, whether, given \( s \ne 1\) and \(p \in (1, \infty)\) such that
\(
\frac{1}{p}-\frac{s}{n}=1-\frac{k}{n},
\)
if \(A(D)\) is elliptic and canceling, 
one has for every \(u \in C^\infty_c(\R^n; V)\), 
\[
  \norm{u }_{\dot{W}^{s, p}} \le C\norm{A(D) u}_{L^1+\dot{W}^{k-s, p}}.
\]

The main ingredient in the proof of theorem~\ref{theoremOrderkBB} is the following variant on theorem~\ref{theoremCocanceling}

\begin{theorem}
\label{theoremVectorL1BB}
Let \(L(D)\) be a linear homogeneous differential operator of order \(k\) on \(\R^n\) from \(E\) to \(F\).
If  \(L(D)\) is cocanceling, then for every \(f \in L^1(\R^n; E)\), one has 
\(f \in\dot{W}^{-1,\frac{n}{n-1}}(\R^n; E)\) if and only if \(L(D)f\in \dot{W}^{-1-k,\frac{n}{n-1}}(\R^n; F)
\).
Moreover, if \(f \in L^1(\R^n; E)\) and \(L(D)f \in\dot{W}^{-1-k,\frac{n}{n-1}}(\R^n; F)\), one has
\[
  \norm{f}_{\dot{W}^{-1, n/(n-1)}} \le C \bigl(\norm{f}_{L^1} + \norm{L(D)f}_{\dot{W}^{-1-k, n/(n-1)}}\bigr),
\]
\end{theorem}

\begin{proof}
The proof follows the lines of the proof of proposition~\ref{propCocancelingSufficient}, it relies on a strengthened version of proposition~\ref{propVSHigherOrder} \cite{VS2008}*{theorem 9}.
\end{proof}

Whereas the sufficiency part of theorem~\ref{theoremVectorL1BB} is much stronger than theorem~\ref{theoremCocanceling}, its proof relies on a difficult construction of J.\thinspace Bourgain and H.\thinspace Brezis \cite{BB2007} while theorem~\ref{theoremCocanceling} relies on proposition~\ref{propVSHigherOrder} that is proved by elementary methods. As it was mentioned for theorem~\ref{theoremVectorL1BB}, the result of J.\thinspace Bourgain and H.\thinspace Brezis has not been extended to other critical Sobolev spaces. 

We can now prove theorem~\ref{theoremOrderkBB}

\begin{proof}[Proof of theorem~\ref{theoremOrderkBB}]
The necessity part follows from theorem~\ref{theoremOrderk}.

For the sufficiency part, choose \(f \in C^\infty_c(\R^n)\) such that 
\[
  \norm{f}_{L^1} + \norm{A(D)u-f}_{\dot{W}^{-1, \frac{n}{n-1}}} \le 2 \norm{A(D) u}_{L^1+\dot{W}^{-1, n/(n-1)}}.
\]
Let \(L(D)\) be the homogeneous differential operator of order \(\ell\) given by proposition~\ref{propositionConstructionL}. Since \(A(D)\) is canceling, \(L(D)\) is cocanceling. In view of theorem~\ref{theoremVectorL1BB}, since \(L(D)f=L(D)\bigl(f-A(D)u\bigr)\), 
\[
\begin{split}
  \norm{f}_{\dot{W}^{-1, \frac{n}{n-1}}} &\le C\bigl(\norm{f}_{L^1}+ \norm{L(D)\bigl(f-A(D)u\bigr)}_{\dot{W}^{-1-\ell, \frac{n}{n-1}}}\bigr)\\
&\le C'\bigl(\norm{f}_{L^1}+ \norm{f-A(D)u}_{\dot{W}^{-1, \frac{n}{n-1}}}\bigr).
\end{split}
\]
We have thus
\[
  \norm{A(D)u}_{\dot{W}^{-1, \frac{n}{n-1}}}  \le C'' \norm{A(D) u}_{L^1+\dot{W}^{-1, n/(n-1)}}.
\]
We conclude by proposition~\ref{propCaldZygmundWk1p} as in the proof of theorem~\ref{theoremOrderk}.
\end{proof}

\section*{Acknowledgment}

The author thanks Ha\"\i m Brezis who has made suggestions on the exposition of this paper and the referee who has read very carefully the manuscript.


\begin{bibdiv}
\begin{biblist}

\bib{Adams}{book}{
   author={Adams, Robert A.},
   title={Sobolev spaces},
   series={Pure and Applied Mathematics}, 
   volume={65},
   publisher={Academic Press}, 
   address={New York~-- London},
   date={1975},
   pages={xviii+268},
}

\bib{Agmon1959}{article}{
   author={Agmon, Shmuel},
   title={The $L_{p}$ approach to the Dirichlet problem. I. Regularity
   theorems},
   journal={Ann. Scuola Norm. Sup. Pisa (3)},
   volume={13},
   date={1959},
   pages={405--448},
}
	
\bib{Agmon}{book}{
   author={Agmon, Shmuel},
   title={Lectures on elliptic boundary value problems},
   series={Van Nostrand Mathematical Studies,
   No. 2},
   publisher={Van Nostrand}, 
   address={Princeton, N.J. -- Toronto -- London},
   date={1965},
   pages={v+291},
}

\bib{Alvino1977}{article}{
   author={Alvino, Angelo},
   title={Sulla diseguaglianza di Sobolev in spazi di Lorentz},
   journal={Boll. Un. Mat. Ital. A (5)},
   volume={14},
   date={1977},
   number={1},
   pages={148--156},
}

\bib{BW}{book}{
   author={Becker, Thomas},
   author={Weispfenning, Volker},
   title={Gr\"obner bases},
   subtitle={A computational approach to commutative algebra},
   publisher={Springer}, 
   address={New York -- Berlin -- Heidelberg},
   date={1993},
  }


\bib{BB2002}{article}{
   author={Bourgain, Jean},
   author={Brezis, Ha{\"{\i}}m},
   title={Sur l'\'equation ${\rm div}\,u=f$},
   journal={C. R. Math. Acad. Sci. Paris},
   volume={334},
   date={2002},
   number={11},
   pages={973--976},
   issn={1631-073X},
}


\bib{BB2003}{article}{
   author={Bourgain, Jean},
   author={Brezis, Ha{\"{\i}}m},
   title={On the equation ${\rm div}\, Y=f$ and application to control of
   phases},
   journal={J. Amer. Math. Soc.},
   volume={16},
   date={2003},
   number={2},
   pages={393--426},
   issn={0894-0347},
}


\bib{BB2004}{article}{
   author={Bourgain, Jean},
   author={Brezis, Ha{\"{\i}}m},
   title={New estimates for the Laplacian, the div-curl, and related Hodge
   systems},
   journal={C. R. Math. Acad. Sci. Paris},
   volume={338},
   date={2004},
   number={7},
   pages={539--543},
   issn={1631-073X},
}


\bib{BB2007}{article}{
   author={Bourgain, Jean},
   author={Brezis, Ha{\"{\i}}m},
   title={New estimates for elliptic equations and Hodge type systems},
   journal={J. Eur. Math. Soc. (JEMS)},
   volume={9},
   date={2007},
   number={2},
   pages={277--315},
   issn={1435-9855},
}

\bib{BourgainBrezisMironescu}{article}{
   author={Bourgain, Jean},
   author={Brezis, Haim},
   author={Mironescu, Petru},
   title={$H^{1/2}$ maps with values into the circle: minimal
   connections, lifting, and the Ginzburg-Landau equation},
   journal={Publ. Math. Inst. Hautes \'Etudes Sci.},
   number={99},
   date={2004},
   pages={1--115},
   issn={0073-8301},
}

\bib{BVS2007}{article}{
   author={Brezis, Ha{\"{\i}}m},
   author={Van Schaftingen, Jean},
   title={Boundary estimates for elliptic systems with \(L^1\)-data},
   journal={Calc. Var. Partial Differential Equations},
   volume={30},
   date={2007},
   number={3},
   pages={369--388},
   issn={0944-2669},
}


\bib{BC}{article}{
   author={Briane, Marc},
   author={Casado-Diaz, Juan},
   title={Estimate of the pressure when its gradient is the divergence of measure. Applications},
   journal={ESAIM Control Optim. Calc. Var.},
   date={October 28 2010},
 doi={10.1051/cocv/2010037},
}


\bib{CZ1952}{article}{
   author={Calder\'on, A. P.},
   author={Zygmund, A.},
   title={On the existence of certain singular integrals},
   journal={Acta Math.},
   volume={88},
   date={1952},
   pages={85--139},
}

\bib{CVS2009}{article}{
   author={Chanillo, Sagun},
   author={Van Schaftingen, Jean},
   title={Subelliptic Bourgain-Brezis estimates on groups},
   journal={Math. Res. Lett.},
   volume={16},
   date={2009},
   number={3},
   pages={487--501},
   issn={1073-2780},
}



\bib{CDDD}{article}{
   author={Cohen, Albert},
   author={Dahmen, Wolfgang},
   author={Daubechies, Ingrid},
   author={DeVore, Ronald},
   title={Harmonic analysis of the space BV},
   journal={Rev. Mat. Iberoamericana},
   volume={19},
   date={2003},
   number={1},
   pages={235--263},
   issn={0213-2230},
}


\bib{CFM2005}{article}{
   author={Conti, Sergio},
   author={Faraco, Daniel},
   author={Maggi, Francesco},
   title={A new approach to counterexamples to \(L^1\) estimates: Korn's
   inequality, geometric rigidity, and regularity for gradients of
   separately convex functions},
   journal={Arch. Ration. Mech. Anal.},
   volume={175},
   date={2005},
   number={2},
   pages={287--300},
   issn={0003-9527},
}


\bib{Ehrenpreis}{article}{
   author={Ehrenpreis, Leon},
   title={A fundamental principle for systems of linear differential
   equations with constant coefficients, and some of its applications},
   conference={
      title={Proc. Internat. Sympos. Linear Spaces},
      address={Jerusalem},
      date={1960},
   },
   book={
      publisher={Jerusalem Academic Press},
      place={Jerusalem},
   },
   date={1961},
   pages={161--174},
}

\bib{F1963}{article}{
  author={de Figueiredo, Djairo Guedes},
  title={The coerciveness problem for forms over vector valued functions},
  journal={Comm. Pure Appl. Math.},
  volume={16}, 
  date={1963},
  pages={63-94},
}

\bib{Franke}{article}{
   author={Franke, Jens},
   title={On the spaces ${\bf F}_{pq}^s$ of Triebel-Lizorkin type:
   pointwise multipliers and spaces on domains},
   journal={Math. Nachr.},
   volume={125},
   date={1986},
   pages={29--68},
   issn={0025-584X},
}

\bib{Gagliardo}{article}{
   author={Gagliardo, Emilio},
   title={Propriet\`a di alcune classi di funzioni in pi\`u variabili},
   journal={Ricerche Mat.},
   volume={7},
   date={1958},
   pages={102--137},
   issn={0035-5038},
}


\bib{Hormander1958}{article}{
   author={H{\"o}rmander, Lars},
   title={Differentiability properties of solutions of systems of
   differential equations},
   journal={Ark. Mat.},
   volume={3},
   date={1958},
   pages={527--535},
   issn={0004-2080},
}

\bib{CKCRAS}{article}{
   author={Kirchheim, Bernd},
   author={Kristensen, Jen},
   title={Automatic convexity of rank--1 convex functions},   
   journal={C. R. Math. Acad. Sci. Paris},
   volume={349},
   date={2011},
   number={7--8},
   pages={407--409},
   issn={1631-073X},
}

\bib{CKPaper}{unpublished}{
   author={Kirchheim, Bernd},
   author={Kristensen, Jen},
   title={On rank one convex functions that are homogeneous of degree one},
   note={in preparation},
}

\bib{Kolyada}{article}{
   author={Kolyada, V. I.},
   title={On the embedding of Sobolev spaces},
   language={Russian},
   journal={Mat. Zametki},
   volume={54},
   date={1993},
   number={3},
   pages={48--71, 158},
   issn={0025-567X},
   translation={
      journal={Math. Notes},
      volume={54},
      date={1993},
      number={3-4},
      pages={908--922 (1994)},
      issn={0001-4346},
   },
}

\bib{Komatsu}{article}{
   author={Komatsu, Hikosaburo},
   title={Resolutions by hyperfunctions of sheaves of solutions of
   differential equations with constant coefficients},
   journal={Math. Ann.},
   volume={176},
   date={1968},
   pages={77--86},
   issn={0025-5831},
}

\bib{LS2005}{article}{
   author={Lanzani, Loredana},
   author={Stein, Elias M.},
   title={A note on div curl inequalities},
   journal={Math. Res. Lett.},
   volume={12},
   date={2005},
   number={1},
   pages={57--61},
   issn={1073-2780},
}

\bib{Mironescu2010}{article}{
   author={Mironescu, Petru},
   title={On some inequalities of Bourgain, Brezis, Maz'ya, and Shaposhnikova
related to \(L^1\) vector fields},
   journal={C. R. Math. Acad. Sci. Paris},
   volume={348},
   date={2010},
   number={9--10},
   pages={513--515},
   issn={1631-073X},
}

\bib{MM2009}{article}{
   author={Mitrea, Irina},
   author={Mitrea, Marius},
   title={A remark on the regularity of the div-curl system},
   journal={Proc. Amer. Math. Soc.},
   volume={137},
   date={2009},
   number={5},
   pages={1729--1733},
   issn={0002-9939},
}


\bib{Nirenberg1959}{article}{
      author={Nirenberg, L.},
       title={On elliptic partial differential equations},
        date={1959},
     journal={Ann. Scuola Norm. Sup. Pisa (3)},
      volume={13},
       pages={115\ndash 162},
}

\bib{Ornstein1962}{article}{
   author={Ornstein, Donald},
   title={A non-equality for differential operators in the \(L_{1}\) norm},
   journal={Arch. Rational Mech. Anal.},
   volume={11},
   date={1962},
   pages={40--49},
   issn={0003-9527},
}



\bib{RunstSickel}{book}{
   author={Runst, Thomas},
   author={Sickel, Winfried},
   title={Sobolev spaces of fractional order, Nemytskij operators, and
   nonlinear partial differential equations},
   series={de Gruyter Series in Nonlinear Analysis and Applications},
   volume={3},
   publisher={Walter de Gruyter \& Co.},
   place={Berlin},
   date={1996},
   pages={x+547},
   isbn={3-11-015113-8},
}


\bib{SchmittWinkler}{thesis}{
  author={Schmitt, B.J.},
  author={Winkler, M.},
  title={On embeddings between $BV$ and $\dot{W}^{s,p}$},
  type={Preprint no. 6},
  date={2000-03-15},
  organization={Lehrstuhl~I f\"ur Mathematik, RWTH Aachen},
}

\bib{Sharafutdinov}{book}{
   author={Sharafutdinov, V. A.},
   title={Integral geometry of tensor fields},
   series={Inverse and Ill-posed Problems Series},
   publisher={VSP},
   place={Utrecht},
   date={1994},
   pages={271},
   isbn={90-6764-165-0},
}

\bib{SickelTriebel}{article}{
   author={Sickel, W.},
   author={Triebel, H.},
   title={H\"older inequalities and sharp embeddings in function spaces of
   $B^s_{pq}$ and $F^s_{pq}$ type},
   journal={Z. Anal. Anwendungen},
   volume={14},
   date={1995},
   number={1},
   pages={105--140},
   issn={0232-2064},
}

\bib{Solonnikov1975}{article}{
   author={Solonnikov, V.A.},
   title={Inequalities for functions of the classes $\dot{W}^{\vec{m}}_p(\R^n)$},
   journal={Zapiski Nauchnykh Seminarov Leningradskogo Otdeleniya Matematicheskogo Instituta im. V. A. Steklova Akademii Nauk SSSR},
   number={27},
   year={1972},
  language={Russian},
   pages={194-210},
   translation={
      journal={J. Sov. Math.},
      volume={3},
      date={1975},
      pages={ 549-564},
   },
}


\bib{Spencer}{article}{
   author={Spencer, D. C.},
   title={Overdetermined systems of linear partial differential equations},
   journal={Bull. Amer. Math. Soc.},
   volume={75},
   date={1969},
   pages={179--239},
   issn={0002-9904},
}


\bib{Stein1970SIDPF}{book}{
   author={Stein, Elias M.},
   title={Singular integrals and differentiability properties of functions},
   series={Princeton Mathematical Series, No. 30},
   publisher={Princeton University Press},
   place={Princeton, N.J.},
   date={1970},
   pages={xiv+290},
}

\bib{Strauss1973}{article}{
   author={Strauss, Monty J.},
   title={Variations of Korn's and Sobolev's equalities},
   conference={
      title={Partial differential equations},
      address={Univ. California, Berkeley, Calif.},
      date={1971}
   },	
   book={
      editor={D. C. Spencer},
      publisher={Amer. Math. Soc.},
      series={Proc. Sympos. Pure Math.},
      volume={23}, 
      place={Providence, R.I.},
   },
   date={1973},
   pages={207--214},
}

\bib{Triebel}{book}{
   author={Triebel, Hans},
   title={Theory of function spaces},
   series={Monographs in Mathematics},
   volume={78},
   publisher={Birkh\"auser Verlag},
   place={Basel},
   date={1983},
   pages={284},
   isbn={3-7643-1381-1},
}

\bib{VS2004BBM}{article}{
   author={Van Schaftingen, Jean},
   title={A simple proof of an inequality of {B}ourgain, {B}rezis and
  {M}ironescu},
   journal={C. R. Math. Acad. Sci. Paris},
   volume={338},
   date={2004},
   number={1},
   pages={23--26},
   issn={1631-073X},
}

\bib{VS2004Divf}{article}{
   author={Van Schaftingen, Jean},
   title={Estimates for \(L^1\)-vector fields},
   journal={C. R. Math. Acad. Sci. Paris},
   volume={339},
   date={2004},
   number={3},
   pages={181--186},
   issn={1631-073X},
}


\bib{VS2004ARB}{article}{
   author={Van Schaftingen, Jean},
   title={Estimates for \(L^1\) vector fields with a second order
   condition},
   journal={Acad. Roy. Belg. Bull. Cl. Sci. (6)},
   volume={15},
   date={2004},
   number={1-6},
   pages={103--112},
   issn={0001-4141},
}

\bib{VS2006BMO}{article}{
   author={Van Schaftingen, Jean},
   title={Function spaces between BMO and critical Sobolev spaces},
   journal={J. Funct. Anal.},
   volume={236},
   date={2006},
   number={2},
   pages={490--516},
}


\bib{VS2008}{article}{
   author={Van Schaftingen, Jean},
   title={Estimates for \(L^1\) vector fields under higher-order
   differential conditions},
   journal={J. Eur. Math. Soc. (JEMS)},
   volume={10},
   date={2008},
   number={4},
   pages={867--882},
   issn={1435-9855},
}

\bib{VS2010}{article}{
   author={Van Schaftingen, Jean},
   title={Limiting fractional and Lorentz space estimates of differential
   forms},
   journal={Proc. Amer. Math. Soc.},
   volume={138},
   date={2010},
   number={1},
   pages={235--240},
   issn={0002-9939},
}

\bib{Yung2010}{article}{
   author={Yung, Po-Lam},
   title={Sobolev inequalities for \((0,q)\) forms on CR manifolds of finite
   type},
   journal={Math. Res. Lett.},
   volume={17},
   date={2010},
   number={1},
   pages={177--196},
   issn={1073-2780},
}

\end{biblist}
\end{bibdiv}
\end{document}